\documentclass[]{article}
\usepackage{amsmath, amsthm, amssymb,verbatim,fullpage,longtable,multicol}
\usepackage{graphicx}
\usepackage{hyperref}
\usepackage[table]{xcolor}
\usepackage[T1]{fontenc}
\usepackage{appendix}
\usepackage[utf8]{inputenc}
\usepackage{authblk}
\newtheorem{theorem}{Theorem}[section]
\newtheorem{lemma}[theorem]{Lemma}
\newtheorem{proposition}[theorem]{Proposition}
\newtheorem{corollary}[theorem]{Corollary}
\newtheorem{definition}{Definition}[section]

\newtheorem{question}[theorem]{Question}
\newtheorem{problem}[theorem]{Problem}
\newtheorem{remark}[theorem]{Remark}
\newtheorem{example}[theorem]{Example}

\definecolor{lightgray}{gray}{0.7}
\definecolor{midgray}{gray}{.9}

\newcommand{\cupdot}{\mathbin{\mathaccent\cdot\cup}}
\def\hroot{\tilde{\alpha}}
\def\B{\mathcal{B}}
\def\C{\mathcal{C}}

\def\A{\mathcal{A}}
\def\D{\mathcal{D}}

\title{The adjoint representation of a Lie algebra and the support of Kostant's weight multiplicity formula}\author[1]{Pamela E. Harris\thanks{pamela.harris@usma.edu. This research was performed while the author held a National Research Council Research Associateship Award at USMA/ARL.}}
\author[2]{Erik Insko\thanks{einsko@fgcu.edu}}
\author[3]{Lauren Kelly Williams\thanks{lwilliams2@mercyhurst.edu }}
\affil[1]{Department of Mathematical Sciences, United States Military Academy}
\affil[2]{Department of Mathematics, Florida Gulf Coast University}
\affil[3]{Department of Mathematics and Computer Systems, Mercyhurst University}

\begin{document}
  \maketitle
\begin{abstract}

Even though weight multiplicity formulas, such as Kostant's formula, exist their computational use is extremely cumbersome. In fact, even in cases when the multiplicity is well
 understood, the number of terms considered in Kostant's formula is factorial in the rank of the Lie algebra and the value of the partition function is unknown. In this paper we address the difficult question: \emph{What are the contributing terms to the multiplicity of the zero weight in the adjoint representation of a finite dimensional Lie algebra?} We describe and enumerate
 the cardinalities of these sets (through linear homogeneous recurrence relations with constant coefficients) for the classical Lie algebras of Type $B$, $C$, and $D$, the Type $A$ case was computed by the first author in \cite{PH}. In addition, we compute the cardinality of the set of contributing terms for non-zero weight spaces in the adjoint representation. In the Type $B$ case, the cardinality of one such non-zero-weight is enumerated by the Fibonacci numbers. We end with a computational proof of a result of Kostant regarding the exponents of the respective Lie algebra for some low rank examples and provide a section with open problems in this area. 

\end{abstract}


\section{Introduction}

In \cite{Narayanan2}, Narayanan proved that the problem of computing Kostka numbers and Littlewood-Richardson coefficients is $\#P$-complete. 
This implies that ``... unless $P = NP$, which is widely disbelieved, there do not exist efficient algorithms that compute these numbers.'' 
Since the Kostka number $K_{\lambda,\mu}$ also can be interpreted as the multiplicity of the weight $\mu$ in the representation of $\mathfrak{sl}_r(\mathbb{C})$ 
with highest weight $\lambda$, which we denote $L(\lambda)$, it is clear that computing weight multiplicities, in much generality, is a computationally complex problem. However, there are cases when computing weight multiplicities can be done in polynomial time. Take for example computing the set of all nonzero Kostka numbers for a particular $\mu$, \cite{Barvinok}.

Though Kostant's formula provides a means to compute weight multiplicities, the computation itself is difficult and time-consuming. 
In fact, even in cases when the multiplicity is well understood, the number of terms appearing in Kostant's formula is exponential in the rank of the Lie algebra 
and the value of the partition function is unknown. 
In this paper we imagine that the value of a partition function (in fact it's $q$-analog) is provided by an oracle. 
We then address the issue of how many terms contribute to the weight multiplicity formula. 

The depth of such approach is appreciated through the easy question: \emph{What is the multiplicity of the zero weight in the adjoint representation?} 
Lie theory provA234599ides the answer almost instantly: the rank of the Lie algebra. In this paper we address the difficult question: 
\emph{What are the contributing terms to the multiplicity of the zero weight in the adjoint representation of a finite dimensional Lie algebra?} 
In Sections \ref{TypeB}, \ref{TypeC}, and \ref{TypeD} we describe and enumerate these supporting sets for the classical Lie algebras of Type $B$, $C$, and $D$, respectively. The Type $A$ case was computed by the first author in \cite{PH}.
We show that the cardinality of the contributing sets satisfy linear homogeneous recurrence relations with constant coefficients. 
Namely we show that the cardinalities of these sets are as follow: \footnote{The sequences of integers for Types $B$, $C$, and $D$ were 
added by the authors to The On-Line Encyclopedia of Integer Sequences (OEIS) as
 \href{http://oeis.org/A232163}{A232163}, \href{http://oeis.org/A232165}{A232165}, and \href{http://oeis.org/A234599}{A234599}, respectively. } \\

Type $A_r$ $(r\geq 2)$:\hspace{.25in}$1, 2, 3, 5, 8, 13, 21, 34, 55, 89, 144, 233, 377, 610, 987, 1597, 2584, 4181, 6765,\ldots$\\

Type $B_r$ $(r\geq 2)$:\hspace{.25in}$2, 5, 10, 22, 49, 106, 231, 506, 1104, 2409, 5262, 11489, 25082, 54766, 119577,\ldots$\\

Type $C_r$ $(r\geq 2)$:\hspace{.25in}$2,3,8,18,37,82,181,392,856,1873,4086,8919,19480,42530,92853,202742,\ldots$\\

Type $D_r$ $(r\geq 4)$:\hspace{.25in}$9,18,35,82,180,385,846,1853,4034,8810,19249,42014,91727,200298,437316,\ldots$\\
\\

This proves that while the number of terms appearing in 
Kostant's formula grows factorially, with the rank of the Lie algebra, the number of terms that contribute non-trivially to the multiplicity 
formula only grow exponentially. In addition, we compute the cardinality of the set of contributing terms for non-zero weight spaces in the adjoint representation. In the Type $B$ case, the cardinality of one such non-zero-weight is enumerated by the Fibonacci numbers. 

This paper ends with some open problems related to Kostant's partition function. We explain how a closed formula for the partition function would lead to a combinatorial proof of a result of Kostant regarding the exponents of the respective Lie algebra. We do provide a proof of this result for some low rank cases.

\section{Background}
Throughout this article we let $G$ be a simple linear algebraic group over $\mathbb C$, $T$ a maximal algebraic torus in $G$ of dimension $r$, and $B$, $T\subseteq B \subseteq G$, a choice of Borel subgroup. Then let $\mathfrak g$, $\mathfrak h$, and $\mathfrak b$ denote the Lie algebras of $G$, $T$, and $B$ respectively. We let $\Phi$ be the set of roots corresponding to $(\mathfrak {g,h})$, and let $\Phi^+\subseteq\Phi$ be the set of positive roots with respect to $\mathfrak b$. Let $\Delta\subseteq\Phi^+$ be the set of simple roots. Denote the set of integral and dominant integral weights by $P(\mathfrak g)$ and $P_+(\mathfrak g)$, respectively. Let $W=Norm_G(T)/T$ denote the Weyl group corresponding to $G$ and $T$, and for any $w\in W$, we let $\ell(w)$ denote the length of $w$.

A finite dimensional complex irreducible representation of $\mathfrak g$ is equivalent to a highest weight representation with dominant integral highest weight $\lambda$, which we denote by $L(\lambda)$. To find the multiplicity of a weight $\mu$ in $L(\lambda)$, we use Kostant's weight multiplicity formula, \cite{KMF}:
\begin{align}
m(\lambda,\mu)=\displaystyle\sum_{\sigma\in W}^{}(-1)^{\ell(\sigma)}\wp(\sigma(\lambda+\rho)-(\mu+\rho))\label{KMF},
\end{align} where $\wp$ denotes Kostant's partition function and $\rho=\frac{1}{2}\sum_{\alpha\in\Phi^+}\alpha$. Recall that Kostant's partition function is the nonnegative integer valued function, $\wp$, defined on $\mathfrak h^*$, by $\wp(\xi)$ = number of ways
$\xi$ may be written as a nonnegative integral sum of positive roots, for $\xi\in\mathfrak{h}^*$.

With the aim of describing the contributing terms of (\ref{KMF}) we introduce.
\begin{definition}\label{definition} For $\lambda,\mu$ dominant integral weights of $\mathfrak g$ define the \emph{Weyl alternation set} to be \begin{center}$\mathcal A(\lambda,\mu)=\{\sigma\in W:\;\wp(\sigma(\lambda+\rho)-(\mu+\rho))>0\}$.\end{center}
\end{definition}

Therefore, $\sigma\in\mathcal A(\lambda,\mu)$ if and only if $\sigma(\lambda+\rho)-(\mu+\rho)$ can be written as a nonnegative integral combination of positive roots. Moreover, in the simple Lie algebra cases, the positive roots are made up of certain nonnegative integral sums of simple roots. Hence we can reduce to $\sigma\in\mathcal A(\lambda,\mu)$ if and only if $\sigma(\lambda+\rho)-(\mu+\rho)$ can be written as a nonnegative integral combination of simple roots.

Of particular interest is describing the elements of the Weyl group which contribute to the multiplicity of the zero weight in the adjoint representation of the classical Lie algebras. That is, we compute the Weyl alternation sets, $\A(\hroot,0)$, for a simple lie algebra $\mathfrak{g}$ where $\hroot$ denotes the highest root. The case of the simple Lie algebra of Type $A$, namely $\mathfrak{sl}_r(\mathbb{C}),$ was completed by Harris in \cite{PH}. In this paper we provide the analogous results for Lie algebras of Type $B$ ($\mathfrak{so}_{2r+1}(\mathbb{C})$), Type $C$ ($\mathfrak{sp}_{2r}(\mathbb{C})$), and Type $D$ $(\mathfrak{so}_{2r}(\mathbb{C}))$. These results are found in Sections~\ref{TypeB}, \ref{TypeC}, and \ref{TypeD}, respectively. 
 
\section{General Results for Classical Lie Algebras}
We begin with some general results regarding the classical Lie algebras. First we give some preliminary information for each of the Lie algebras we consider, for notation see \cite{GW}.\\

\noindent Type $A_r$ ($\mathfrak{sl}_{r}(\mathbb C)$): Let  $r\geq 1$ and let $\alpha_i=\varepsilon_i-\varepsilon_{i+1}$ for $1\leq i\leq r$. Then $\Delta = \{\alpha_i \ | \ 1 \leq i \leq r \}$, is a set of simple roots. The associated set of positive roots is $\Phi^+=\{\varepsilon_i-\varepsilon_j:1\leq i<j\leq r \}$, where the highest root is $\tilde{\alpha}=\alpha_1+\alpha_2+\cdots+\alpha_r$ and $\rho=\frac{1}{2}\sum_{i=1}^{r}i(r-i+1)\alpha_i$. \\

\noindent
Type $B_r$ ($\mathfrak{so}_{2r+1}(\mathbb C)$): Let $r\geq 2$ and let $\alpha_i=\varepsilon_i-\varepsilon_{i+1}$ for $1\leq i\leq r-1$ and $\alpha_r=\varepsilon_r$. Then $\Delta = \{\alpha_i \ | \ 1 \leq i \leq r \}$, is a set of simple roots. The associated set of positive roots is $\Phi^+=\{\varepsilon_i-\varepsilon_j,\varepsilon_i+\varepsilon_j:1\leq i<j\leq r \}\cup\{\varepsilon_i:1\leq i\leq r\}$, where the highest root is $\tilde{\alpha}=\alpha_1+2\alpha_2+\cdots+2\alpha_r$ and $\rho=\frac{1}{2}\sum_{i=1}^{r}i(2r-i)\alpha_i$. \\

\noindent
Type $C_r$ ($\mathfrak{sp}_{2r}(\mathbb C)$): Let $r\geq 3$ and let $\alpha_i=\varepsilon_i-\varepsilon_{i+1}$ for $1\leq i\leq r-1$ and $\alpha_r=2\varepsilon_r$. Then $\Delta = \{\alpha_i \ | \ 1 \leq i \leq r \}$, is a set of simple roots. The associated set of positive roots is $\Phi^+=\{\varepsilon_i-\varepsilon_j,\varepsilon_i+\varepsilon_j:1\leq i<j\leq r \}\cup\{2\varepsilon_i:1\leq i\leq r\}$, where the highest root is $\tilde{\alpha}=\alpha_1+2\alpha_2+\cdots+2\alpha_{r-1}+\alpha_r$ and $\rho=\frac{1}{2}\sum_{i=1}^{r-1}i(2r-i+1)\alpha_i+\frac{r(r+1)}{4}\alpha_r$.\\ 

\noindent
Type $D_r$ ($\mathfrak{so}_{2r}(\mathbb C)$): Let $r\geq 4$ and let $\alpha_i=\varepsilon_i-\varepsilon_{i+1}$ for $1\leq i\leq r-1$ and $\alpha_r=\varepsilon_{r-1}+\varepsilon_r$. Then $\Delta = \{\alpha_i \ | \ 1 \leq i \leq r \}$, is a set of simple roots. The associated set of positive roots is $\Phi^+=\{\varepsilon_i-\varepsilon_j,\;\varepsilon_i+\varepsilon_j\ | \ 1\leq i<j\leq r\}$, where the highest root is $\tilde{\alpha}= \varepsilon_1 + \varepsilon_2=\alpha_1+2\alpha_2+\cdots+2\alpha_{r-2}+\alpha_{r-1}+\alpha_r$ and $\rho = \frac{1}{2} \sum_{\alpha \in \Phi^+} \alpha = (r-1)\varepsilon_1 + (r-2)\varepsilon_2 + (r-3)\varepsilon_3  + \cdots 2\varepsilon_{r-2} + \varepsilon_{r-1} =\frac{1}{2}\sum_{i=1}^{r-2}2(ir-\frac{i(i+1)}{2})\alpha_i+\frac{r(r-1)}{4}(\alpha_{r-1}+\alpha_r)$.

\begin{lemma}\label{simpleonhroot}
The following simple transpositions do not fix the highest root in each respective Lie type.
\begin{itemize}
 \item In type $A_r$: $s_1(\hroot) = \hroot - \alpha_1$ and $s_r (\hroot) = \hroot - \alpha_r$,
\item In type $B_r$: $s_2(\hroot) = \hroot - \alpha_2$,
\item In type $C_r$: $s_1(\hroot) = \hroot - 2 \alpha_1$,
\item In type $D_r$: $s_2(\hroot) = \hroot - \alpha_2$.
\end{itemize}
The rest of simple reflections fix the highest root $s_i(\hroot) = \hroot$.
\end{lemma}
\begin{proof}
The lemma follows from facts about how simple transpositions act on simple roots.\\

\noindent
Type $A_r$: For $1 \leq i \leq r-1$ we have $s_i(\alpha_i) = -\alpha_i$, $s_i (\alpha_{i-1}) = \alpha_{i-1}+ \alpha_i$, and $s_i (\alpha_{i+1}) = \alpha_i + \alpha_{i+1}.$
For $i=r$, we have that $s_r(\alpha_r)=-\alpha_r$ and $s_r(\alpha_{r-1})=\alpha_{r-1}+\alpha_r.$ The highest root in this case is $\hroot=\alpha_1+\cdots+\alpha_r$.

For $2\leq i\leq r-1$
\[s_i(\hroot)
=\alpha_1+\alpha_2+\cdots+\alpha_{i-2}+(\alpha_{i-1}+\alpha_i)+(-\alpha_i)+(\alpha_i+\alpha_{i+1})+\alpha_{i+2}+\cdots+\alpha_r\\
=\hroot.\]
Finally observe that
\[s_1(\hroot)
=(-\alpha_1)+(\alpha_1+\alpha_2)+\alpha_3+\cdots+\alpha_r
=\hroot-\alpha_1,\]
and 
\[s_r(\hroot)
=\alpha_1+\cdots+\alpha_{r-2}+(\alpha_{r-1}+\alpha_r)+(-\alpha_r)
=\hroot-\alpha_r.\]

\noindent
Type $B_r$: For $1 \leq i \leq r-1$ we have $s_i(\alpha_i) = -\alpha_i$, $s_i (\alpha_{i-1}) = \alpha_{i-1}+ \alpha_i$, and $s_i (\alpha_{i+1}) = \alpha_i + \alpha_{i+1}.$
For $i =r$ we have that $s_{r}(\alpha_r) = -\alpha_r $ and $s_r (\alpha_{r-1}) = \alpha_{r-1}+ 2 \alpha_r.$ The highest root in this case is $\hroot=\alpha_1+2\alpha_2+\cdots+2\alpha_r$.

Thus we compute that \[s_1(\hroot) = -\alpha_1+ 2 \alpha_1+ 2\alpha_2+ \cdots +2 \alpha_r=\hroot, \] and
\[s_2( \hroot ) = \alpha_1 - 2 \alpha_2 + 3 \alpha_2 + 2 \alpha_3+ \cdots + 2 \alpha_r = \alpha_1+\alpha_2 + 2 \alpha_3 + \cdots + 2 \alpha_r = \hroot - \alpha_2.\]
For $ 3\leq i\leq r-1 $ we compute that
\[s_i(\hroot) = \alpha_1 + 2 \alpha_2 + \cdots + 2 \alpha_{i-2}+2(\alpha_{i-1}+\alpha_i)-2\alpha_i+2(\alpha_i+\alpha_{i+1})+ 2 \alpha_{i+2} + \cdots + 2 \alpha_r=\hroot.  \]  
Finally, \[s_r( \hroot) = \alpha_1+ 2 \alpha_2+ \cdots + 2 \alpha_{r-2}+2(\alpha_{r - 1}+ 2\alpha_r)-2\alpha_r=\hroot.  \]  

\noindent
Type $C_r$: For $1 \leq i \leq r$ we have $s_i(\alpha_i) = -\alpha_i $, $ s_i (\alpha_{i-1}) = \alpha_{i-1}+ \alpha_i$. For $1\leq i\leq r-2$, $ s_i (\alpha_{i+1}) = \alpha_i + \alpha_{i+1},$ while $s_{r-1} (\alpha_{r}) = 2\alpha_{r-1}+ \alpha_r .$ The highest root in this case is $\hroot=2\alpha_1+\cdots+2\alpha_{r-1}+\alpha_r$.

Thus we compute that \[s_1(\hroot) = 2(-\alpha_1)+ 2 (\alpha_1+\alpha_2)+ 2\alpha_3+ \cdots + 2\alpha_{r-1}+\alpha_r=2\alpha_2+\cdots+2\alpha_{r-1}+\alpha_r=\hroot-2\alpha_1, \]
For $ 2\leq i\leq r-2 $ we compute that
\[s_i(\hroot) = 2\alpha_1 + 2 \alpha_2 + \cdots + 2 \alpha_{i-2}+2(\alpha_{i-1}+\alpha_i)-2\alpha_i+2(\alpha_i+\alpha_{i+1})+ 2 \alpha_{i+2} + \cdots + 2 \alpha_{r-1}+\alpha_r=\hroot.  \]  
Finally, 
\[s_{r-1}( \hroot) = 2\alpha_1+ 2 \alpha_2+ \cdots + 2 \alpha_{r-3}+2(\alpha_{r - 2}+ \alpha_{r-1})-2\alpha_{r-1}+(2\alpha_{r-1}+\alpha_r)=\hroot,  \]  
and
\[s_r( \hroot) = 2\alpha_1+ 2 \alpha_2+ \cdots + 2 \alpha_{r-2}+2(\alpha_{r - 1}+ \alpha_r)-\alpha_r=\hroot.  \]

\noindent
Type $D_r$: For $1\leq i\leq r$, $s_i(\alpha_i)=-\alpha_i$. If $1\leq i<j\leq r-1$ with $|i-j|=1$ or if $i=r-2$ and $j=r$, then $s_i(\alpha_j)=s_j(\alpha_i)=\alpha_i+\alpha_j$.
For $i=r-1$ or $i=r$ we have that $s_{r-1}(\alpha_{r})=\alpha_r$ and $s_{r}(\alpha_{r-1})=\alpha_{r-1}.$ The highest root in this case is $\hroot=\alpha_1+ 2 \alpha_2+ \cdots + 2 \alpha_{r-2}+ \alpha_{r-1}+\alpha_r$.

Thus we compute that

\[s_1(\hroot)
= -\alpha_1+2(\alpha_1+\alpha_2)+ 2\alpha_3+\cdots+ 2 \alpha_{r-2} + \alpha_{r-1}+\alpha_r 
=\hroot,\]
and \[s_2(\hroot)
= (\alpha_1+\alpha_2)- 2\alpha_2 + 2(\alpha_2+ \alpha_3)+2\alpha_4+ \cdots + 2 \alpha_{r-2} + \alpha_{r-1}+\alpha_r 
=\hroot - \alpha_2.\]
For $3\leq i\leq r-3$,
\[s_i(\hroot)
= \alpha_1+2\alpha_2 +\cdots+ 2\alpha_{i-2}+2(\alpha_{i-1}+ \alpha_i)-2\alpha_i+2(\alpha_{i}+\alpha_{i+1})+2\alpha_{i+2}+ \cdots + 2 \alpha_{r-2} + \alpha_{r-1}+\alpha_r 
=\hroot.\]
Finally observe that
\[s_{r-1}(\hroot)= \alpha_1+2\alpha_2 + \cdots + 2 \alpha_{r-3}+2(\alpha_{r-2}+\alpha_{r-1})-\alpha_{r-1} + \alpha_{r} 
=\hroot,\]

\[s_{r}(\hroot)= \alpha_1+2\alpha_2 + \cdots + 2 \alpha_{r-3}+2(\alpha_{r-2}+\alpha_{r})+\alpha_{r-1} - \alpha_{r} 
=\hroot,\]
and 
\[s_{r-2}(\hroot)
= \alpha_1+2\alpha_2 + \cdots + 2 \alpha_{r-4}+2(\alpha_{r-3}+\alpha_{r-2})-2\alpha_{r-2} + (\alpha_{r-2}+\alpha_{r-1})+(\alpha_{r-2}+\alpha_{r}) 
=\hroot.\]

\end{proof}

\begin{figure}[h]
 \begin{picture}(100,50)(-105,-100)

\color{black}
\multiput(95,-70)(50,0){2}{\circle{5}}
\put(97,-70){\line(1,0){45}}
\put(90,-80){$\alpha_1$}
\put(140,-80){$\alpha_2$}

\end{picture}
\caption{Dynkin diagram of the root system $A_2$} \label{Figure:A_2}
\end{figure}
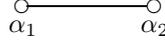

\begin{lemma} \label{A_2}
When the Dynkin diagram of $\alpha_i$ and $\alpha_{i+1}$ embeds into that of $A_2$ (Figure~\ref{Figure:A_2})
the products of $s_i$ and $s_{i+1}$ have the following effect on $2 \rho$

\begin{align*}
s_i (2 \rho) = 2 \rho - 2 \alpha_i &\text{ and } s_{i+1}(2 \rho) = 2 \rho- 2\alpha_{i+1}, \\
s_{i+1} s_{i} ( 2\rho) &= 2 \rho  - 2\alpha_{i}- 4 \alpha_{i+1},\\
s_i s_{i+1} ( 2\rho) &= 2 \rho - 4 \alpha_{i} - 2\alpha_{i+1},\\
s_i s_{i+1} s_i (2 \rho) &= 2 \rho - 4 \alpha_i - 4\alpha_{i+1}.
\end{align*}
\end{lemma}

\begin{proof}
A simple reflection $s_i$ maps $\alpha_i$ to $-\alpha_i$ and permutes all of the other positive roots.  
Thus $s_i (2 \rho) = 2\rho - 2 \alpha_i$ for any simple transposition in any Lie type.  This proves the first two equations.

Next we note that $s_{i+1}(\alpha_{i}+\alpha_{i+1}) = \alpha_i$; hence $s_i s_{i+1}(\alpha_{i}+\alpha_{i+1}) = -\alpha_i$ and $s_{i+1} s_{i}(\alpha_{i}) = s_{i+1}(-\alpha_{i}) = -\alpha_{i} - \alpha_{i+1}$.
Since $s_i s_{i+1}$ is a length two element of $W$, it maps only two positive roots to negative roots.  
Thus $s_{i} s_{i+1}$ permutes all of the  other positive roots.  We conclude that $s_{i}s_{i+1}(2 \rho) =2 \rho - 4 \alpha_i - 2 \alpha_{i+1}$.
The same calculation shows the claim for $s_{i+1} s_{i}$.  This proves the next two equations.

Since the reflection $s_i s_{i+1} s_i$ has length $3$, we know it maps three positive roots to negative roots.  We calculate
that it maps $\alpha_i$ to $-\alpha_{i+1}$, $\alpha_{i+1}$ to $-\alpha_{i}$ and $\alpha_{i}+\alpha_{i+1}$ to $-\alpha_i-\alpha_{i+1}$. 
Hence the reflection $s_i s_{i+1} s_i$ must map all of the other positive roots to other positive roots.  We conclude that $s_i s_{i+1} s_i (2 \rho) = 2 \rho - 4 \alpha_i - 4 \alpha_{i+1}$.
\end{proof}

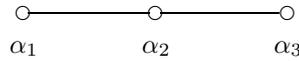
\begin{figure}[h!]
 \begin{picture}(100,50)(-85,-100)
\color{black}
\multiput(95,-70)(50,0){3}{\circle{5}}
\put(97,-70){\line(1,0){46}}
\put(147,-70){\line(1,0){45}}
\put(90,-85){$\alpha_1$}
\put(140,-85){$\alpha_2$}
\put(190,-85){$\alpha_3$}

\end{picture}
\caption{Dynkin diagram of the root system $A_3$} \label{Figure:A_3}
\end{figure}

\begin{lemma} \label{A_3}
If the Dynkin diagram of $\alpha_i$, $\alpha_{i+1}$, and $\alpha_{i+2}$ embeds into the Dynkin diagram of $A_3$ (Figure~\ref{Figure:A_3}), then
the elements $s_i s_{i+1} s_{i+2}$,  $s_{i+2} s_{i+1} s_{i}$,$ s_is_{i+2}s_{i+1}$, and $s_{i+1}s_{i}s_{i+2}$ act on the sum of positive roots as follows:
\begin{align*}
s_i s_{i+1} s_{i+2}(2 \rho) &= 2\rho  - 6 \alpha_{i}-4 \alpha_{i+1} - 2\alpha_{i+2}, \\
s_{i+2} s_{i+1} s_{i}(2 \rho) &= 2 \rho - \alpha_{i} - 4 \alpha_{i+1} - 6 \alpha_{i+2},\\
s_is_{i+2}s_{i+1} ( 2 \rho) &= 2 \rho - 4 \alpha_i -2 \alpha_{i+1} - 4 \alpha_{i+2}, \\
s_{i+1}s_{i}s_{i+2} (2 \rho) &= 2 \rho - 2 \alpha_i -6 \alpha_{i+1} -2 \alpha_{i+2}. 
\end{align*}
\end{lemma}

\begin{proof}
When the Dynkin diagram of the consecutive roots looks like that of $A_3$ we calculate that three roots $\alpha_{i+2}, \alpha_{i+1} + \alpha_{i+2}, \alpha_{i} + \alpha_{i+1}+ \alpha_{i+2}$ get mapped to 
$-\alpha_{i}-\alpha_{i+1}- \alpha_{i+2}$, $-\alpha_i - \alpha_{i+1}$, and $-\alpha_i$ by $s_i s_{i+1} s_{i+2}$.
The roots $\alpha_{i}, \alpha_{i+1} + \alpha_{i}, \alpha_{i} + \alpha_{i+1}+ \alpha_{i+2}$ get mapped to $-\alpha_{i}-\alpha_{i+1}- \alpha_{i+2}$, $-\alpha_{i+1} - \alpha_{i+2}$, and $-\alpha_{i+1}$ by 
$s_{i+2} s_{i+1} s_{i}$.  
The word $s_is_{i+2} s_{i+1}$ maps the positive roots $\alpha_i+\alpha_{i+1}, \alpha_{i+1}+ \alpha_{i+2},$ and $\alpha_{i+1}$ to the negative roots
$ -\alpha_i$ , $-\alpha_{i+2}$, and $-\alpha_{i}-\alpha_{i+1}-\alpha_{i+2}$ respectively.
The word $s_{i+1}s_is_{i+2}$ maps $\alpha_i, \alpha_{i+2},$ and $\alpha_i + \alpha_{i+1} + \alpha_{i+2}$ to $-\alpha_i -\alpha_{i+1}, - \alpha_{i+1}-\alpha_{i+2},$ and $-\alpha_{i+1}$ respectively, and it
permutes the rest of the positive roots. 
\end{proof}

\begin{figure}[h!]
 \begin{picture}(100,50)(-60,-100)
\color{black}
\multiput(95,-70)(50,0){4}{\circle{5}}
\put(97,-70){\line(1,0){45}}
\put(147,-70){\line(1,0){45}}
\put(197,-70){\line(1,0){45}}
\put(90,-85){$\alpha_1$}
\put(140,-85){$\alpha_2$}
\put(190,-85){$\alpha_3$}
\put(230,-85){$\alpha_4$}
\end{picture}
\caption{Dynkin diagram of the root system $A_4$} \label{Figure:A_4}
\end{figure}
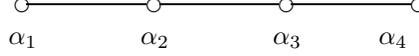

\begin{lemma}\label{A_4}
If the Dynkin diagram of $\alpha_{i},$ $\alpha_{i+1},$ $ \alpha_{i+2}$, and $\alpha_{i+3}$ embeds into $A_4$ (Figure~\ref{Figure:A_4}), then no $\sigma$  
containing all of the simple transpositions $s_i$, $s_{i+1}$, $s_{i+2}$, and $s_{i+3}$
at least once, is in the support of the Weyl alternation set.  
\end{lemma}
\begin{proof}
By Lemma \ref{A_3}, the only length-three product of $s_i$, $s_{i+1}$ and $s_{i+2}$ that is in the support of the Weyl alternation set is 
$s_is_{i+2}s_{i+1}$.  To obtain a word $\sigma $ with all four simple transpositions, we can either multiply on the left or right by $s_{i+3}$.
However, if we multiply by $s_{i+3}$ on the left we get $s_{i+3}s_is_{i+2}s_{i+1} = s_i s_{i+3}s_{i+2}s_{i+1}$. This contains $s_{i+3}s_{i+2}s_{i+1}$ which is not in the support by Lemma \ref{A_3}.
If we multiply on the right by $s_{i+3}$ we obtain $s_is_{i+2}s_{i+1}s_{i+3}$.  By Lemma \ref{A_3}, $s_{i+2}s_{i+1}s_{i+3}$ is not in the support because
$s_{i+2}s_{i+1}s_{i+3} (2 \rho) = 2 \rho - 2 \alpha_{i+1} -6 \alpha_{i+2} -2 \alpha_{i+3}$.  An analogous argument shows that $s_{i+1}s_{i+3}s_{i+2}$ can not be extended to a product containing $s_i$
that is in the support. 
\end{proof}

\begin{figure}[h]
 \begin{picture}(100,50)(-105,-100)

\color{black}
\multiput(95,-70)(50,0){2}{\circle{5}}
\put(97,-68){\line(1,0){46}}
\put(97,-72){\line(1,0){46}}
\put(125,-70){\line(-1,1){10}}
\put(125,-70){\line(-1,-1){10}}

\put(90,-80){$\alpha_1$}
\put(140,-80){$\alpha_2$}

\end{picture}
\caption{Dynkin diagram of the root system $B_2$} \label{Figure:B_2}
\end{figure}
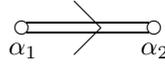

\begin{lemma} \label{B_2}
When the Dynkin diagram of $\alpha_i$ and $\alpha_{i+1}$ embeds into the Dynkin diagram of type $B_2$ (Figure~\ref{Figure:B_2}) or when $i = r-1$ in type $B_r$, we have the following
 \[ s_{i}s_{i+1}(2 \rho) = 2 \rho -4 \alpha_{i}-2 \alpha_{i+1} \text{ and }  s_{i+1} s_{i} (2 \rho) = 2 \rho - 2 \alpha_{i}- 6 \alpha_{i+1}.  \]
\end{lemma}
\begin{proof}
When $i = r-1$, we calculate that $s_{r-1}s_r (\alpha_{r-1}+ 2 \alpha_r ) = s_{r-1}(\alpha_{r-1}) = - \alpha_{r-1}$ and $s_{r-1}s_r(\alpha_r) = s_{r-1}(-\alpha_r) = -\alpha_{r-1}- \alpha_r$. 
Again, $ s_{r-1}s_r$ is a length two element, so these are the only roots which get mapped to negative roots.  Thus $s_{r-1}s_r(2 \rho) = 2 \rho -4 \alpha_{r-1}-2 \alpha_r$. 

 To show that  $ s_r s_{r-1} (2 \rho) = 2 \rho - 6 \alpha_r - 2 \alpha_{r-1} $ we note that $s_r s_{r-1} ( \alpha_{r-1}) = s_r (- \alpha_{r-1} ) = -\alpha_{r-1} - 2 \alpha_r$ and 
$s_r s_{r-1} ( \alpha_{r-1}+ \alpha_r) = s_r ( \alpha_r)= -\alpha_r$.  These are the only two roots which get sent to negative roots.  So the other positive roots must be permuted by $s_r s_{r-1}$.
We conclude that  $ s_r s_{r-1} (2 \rho) = 2 \rho - 6 \alpha_r - 2 \alpha_{r-1}$.
\end{proof}

\begin{figure}[h]
 \begin{picture}(100,50)(-105,-100)

\color{black}
\multiput(95,-70)(50,0){2}{\circle{5}}
\put(97,-68){\line(1,0){46}}
\put(97,-72){\line(1,0){46}}
\put(115,-70){\line(1,1){10}}
\put(115,-70){\line(1,-1){10}}

\put(90,-80){$\alpha_1$}
\put(140,-80){$\alpha_2$}

\end{picture}
\caption{Dynkin diagram of the root system $C_2$} \label{Figure:C_2}
\end{figure}
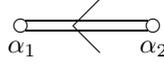

\begin{lemma} \label{C_2}
When the Dynkin diagram of $\alpha_i$ and $\alpha_{i+1}$ embeds into the Dynkin diagram of type $C_2$ (Figure~\ref{Figure:C_2}) or when $i = r-1$ in type $C_r$, we have the following
 \[ s_{i}s_{i+1}(2 \rho) = 2 \rho -6\alpha_{i}-2 \alpha_{i+1} \text{ and }  s_{i+1} s_{i} (2 \rho) = 2 \rho - 2 \alpha_{i} - 4 \alpha_{i+1}. \]
\end{lemma}
\begin{proof}
When $i = r-1$, we calculate that $s_{r-1}s_r (\alpha_{r} ) = s_{r-1}(-\alpha_{r}) = - 2\alpha_{r-1}-\alpha_r$ and $s_{r-1}s_r(\alpha_{r-1}+\alpha_r) = s_{r-1}(\alpha_{r-1}) = -\alpha_{r-1}$. 
Again, $ s_{r-1}s_r$ is a length two element, so these are the only roots which get mapped to negative roots.  Thus $s_{r-1}s_r(2 \rho) = 2 \rho -6 \alpha_{r-1}-2 \alpha_r$. 

 To show that  $ s_r s_{r-1} (2 \rho) = 2 \rho - 2 \alpha_{r-1} - 4 \alpha_{r} $ we note that $s_r s_{r-1} ( \alpha_{r-1}) = s_r (- \alpha_{r-1} ) = -\alpha_{r-1} -  \alpha_r$ and 
$s_r s_{r-1} ( 2\alpha_{r-1}+ \alpha_r) = s_r ( \alpha_r)= -\alpha_r$.  These are the only two roots which get sent to negative roots.  So the other positive roots must be permuted by $s_r s_{r-1}$.
We conclude that  $ s_r s_{r-1} (2 \rho) = 2 \rho - 2 \alpha_{r-1} - 4 \alpha_{r}$.
\end{proof}

Lemmas \ref{A_3} and \ref{A_4} allow us to identify a set of Weyl group elements which are not in the Weyl alternation set for any classical type.  We record this set now for ease of reference in the type specific proofs presented in the following sections.  
\begin{lemma}\label{classic}
Let $\sigma \in W_r$ be a Weyl group element in any classical Lie type.  If $\sigma$ contains a subword of the form 
\[ s_i s_{i+1} s_{i+2}, \ 
s_{i+2} s_{i+1} s_{i}, \text{ or }
s_{i+1}s_{i}s_{i+2}  \]
or any product of four consecutive simple reflections $s_i, s_{i+1}, s_{i+2}, s_{i+3}$ (in any order) then $\sigma$ is not in the Weyl alternation set $\A(\hroot,0)$.
\end{lemma}
\begin{proof}
Recall that in any classical Lie algebra, when the highest root $\hroot$ is written as a linear combination of the simple roots, the coefficients are either 1 or 2.  
Thus the coefficients of $2 \hroot$ are either 2 or 4.  Any Weyl group element $\sigma \in W_r $ will either fix $\hroot$ or map it to a shorter root or possibly negative root.    
This means that the coefficients of $\sigma( 2 \hroot)$ are at most 4.  It follows that if $\sigma(2 \rho)-2\rho $ has a coefficient less than  $-4$ when written as a linear combination of simple roots,
 then $\sigma (2\hroot+2\rho)-2 \rho$ contains a negative coefficient as well.  

Lemma \ref{A_3} shows that $ s_i s_{i+1} s_{i+2} (2\rho)-2\rho$ contains a term $-6 \alpha_i$,  $ s_{i+2} s_{i+1} s_{i} (2\rho)-2\rho$ contains $-6 \alpha_{i+2}$, and $s_{i+1}s_{i}s_{i+2} (2\rho)-2\rho$ contains a $-6 \alpha_{i+1}$.  Thus any element $\sigma$ containing one of these subwords will not be in the Weyl alternation set.  
Lemma \ref{A_4} shows that any $\sigma$ containing a product of four consecutive simple reflections $s_i, s_{i+1}, s_{i+2}, s_{i+3}$ (in any order) must contain either $s_i s_{i+1} s_{i+2}, \ 
s_{i+2} s_{i+1} s_{i}, \text{ or }
s_{i+1}s_{i}s_{i+2} $.  Hence, no such $\sigma$ will be in the Weyl alternation set of any classical Lie algebra.
\end{proof}

\section{Type $B$}\label{TypeB}

When we consider the Lie algebra of type $B$ and rank $r$ we denote the Weyl alternation set as follows:$$\B_r:= \mathcal A(\tilde{\alpha},0)=\{\sigma\in W:\wp(\sigma(\hroot+\rho)-\rho)>0\},$$

where $W$ is the Weyl group and $\hroot$ denotes the highest root of $B_r$. Namely $\hroot=\alpha+2\alpha_2+\cdots+2\alpha_r.$ 

\begin{figure}[h]
\begin{picture}(100,75)(-10,-100)
\thicklines
\thinlines
\color{black}
\multiput(95,-70)(50,0){6}{\circle{5}}
\put(97,-70){\line(1,0){46}}
\put(147,-70){\line(1,0){46}}

\multiput(210,-70)(10,0){3}{\circle{2}}
\put(247,-70){\line(1,0){46}}
\put(297,-69){\line(1,0){46}}
\put(297,-71){\line(1,0){46}}
\put(325,-70){\line(-1,1){10}}
\put(325,-70){\line(-1,-1){10}}
\put(90,-85){$\alpha_1$}
\put(140,-85){$\alpha_2$}
\put(190,-85){$\alpha_3$}
\put(240,-85){$\alpha_{r-2}$}
\put(290,-85){$\alpha_{r-1}$}
\put(340,-85){$\alpha_{r}$}
\end{picture}
\caption{Dynkin diagram of the root system $B_r$} \label{Figure:B_r}
\end{figure}
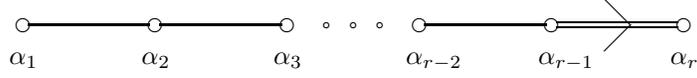

In order to illustrate the complexity in computing weight multiplicities we present a detailed example.

\begin{example}\label{example} We will use Kostant's weight multiplicity formula to compute the multiplicity of the zero-weight in the adjoint representation of $\mathfrak{so}_{7}(\mathbb{C})$. In this process we will compute the Weyl alternation set $\B_3$. First note that the Weyl group, $W,$ corresponding to the Lie algebra of type $B_r$ has order $2^r r!$ Hence, when $r=3$ the Weyl group has order 48. This means that Kostant's weight multiplicity formula will be an alternating sum consisting of 48 terms. 

We begin by considering the term corresponding to the identity element of $W.$ First notice that $1(\hroot+\rho)-\rho=\hroot$ and now we must compute the value of Kostant's partition function. To compute the number of ways to express $\hroot$ as a sum of positive roots we use parenthesis to denote which positive roots we are using when expressing $\hroot$ as a nonnegative integral combination of positive roots. In this way we can see that
\begin{align*}
\hroot&=(\alpha_1)+2(\alpha_2)+2(\alpha_3)\\
&=(\alpha_1)+2(\alpha_2+\alpha_3)\\
&=(\alpha_1)+(\alpha_2)+(\alpha_3)+(\alpha_2+\alpha_3)\\
&=(\alpha_1)+(\alpha_2)+(\alpha_2+2\alpha_3)\\
&=(\alpha_1+\alpha_2)+(\alpha_2)+2(\alpha_3)\\
&=(\alpha_1+\alpha_2)+(\alpha_2+\alpha_3)+(\alpha_3)\\
&=(\alpha_1+\alpha_2+\alpha_3)+(\alpha_2)+(\alpha_3)\\
&=(\alpha_1+\alpha_2+\alpha_3)+(\alpha_2+\alpha_3)\\
&=(\alpha_1+2\alpha_2)+2(\alpha_3)\\
&=(\alpha_1+\alpha_2+2\alpha_3)+(\alpha_2)\\
&=(\alpha_1+2\alpha_2+2\alpha_3).
\end{align*} 
Thus $\wp(1(\hroot+\rho)-\rho)=11$. Table~\ref{Example:B3} summarizes these computations for all 48 elements of the Weyl group. Observe that of the 48 elements of the Weyl group only 5 elements, namely $1$, $s_1$, $s_2$, $s_3$, and $s_3s_1$, contribute a positive partition function value. Thus $\B_3=\{1, s_1, s_2, s_3, s_3s_1\}$. It is worth remarking again that as the rank of the Lie algebra increases the number of terms grows exponentially, and thus it is more evident that it is essential to know which elements are contributing nonzero terms to the alternating sum. 

Now we can finally compute the multiplicity of the zero-weight in the adjoint representation by reducing the sum to only the contributing terms. Thus

\[m(\hroot,0)=\sum_{\sigma\in W}(-1)^{\ell(\sigma)}\wp(\sigma(\hroot+\rho)-\rho)=\sum_{\sigma\in\B_3}(-1)^{\ell(\sigma)}\wp(\sigma(\hroot+\rho)-\rho)=11-4-1-5+2=3,\]
which is the rank of the Lie algebra $\mathfrak{so}_7(\mathbb{C}),$ as we expected.

\begin{table}[htp]

\centering
\caption{Data for Lie algebra of Type $B_3$.}\label{Example:B3}
\rowcolors{1}{}{midgray}
 \begin{longtable}{|l|c|c|c|}
\hline
$\sigma\in W$						& $\ell(\sigma)$	  & $\sigma(\hroot+\rho)-\rho$			&$\wp(\sigma(\hroot+\rho)-\rho)$\\
\hline
$1$	     							& 	0	  &	$\alpha_1+2\alpha_2+2\alpha_3$		& 11	\\
$s_1$	     						& 	1	  &	$2\alpha_2+2\alpha_3$				& 4	\\
$s_2$	     						& 	1	  &	$\alpha_1+2\alpha_3$				& 1	\\
$s_3$	     						& 	1	  &	$\alpha_1+2\alpha_2+\alpha_3$		& 5	\\

$s_1s_2$	     						& 	2	  &	$-2\alpha_1+2\alpha_3$				& 0	\\
$s_2s_1$	     						& 	2	  &	$-\alpha_2+2\alpha_3$				& 0	\\
$s_2s_3$	     						& 	2	  &	$\alpha_1-\alpha_2+\alpha_3$			& 0	\\
$s_3s_1$	     						& 	2	  &	$2\alpha_2+\alpha_3$				& 2	\\
$s_3s_2$	     						& 	2	  &	$\alpha_1-3\alpha_3$				& 0	\\

$s_1s_2s_1$	     					& 	3	  &	$-2\alpha_1-\alpha_2+2\alpha_3$		& 0	\\
$s_1s_2s_3$	     					& 	3	  &	$-3\alpha_1-\alpha_2+\alpha_3$		& 0	\\
$s_2s_3s_1$	     					& 	3	  &	$-2\alpha_2+\alpha_3$				& 0	\\
$s_2s_3s_2$	     					& 	3	  &	$\alpha_1-3\alpha_2-3\alpha_3$		& 0	\\
$s_3s_1s_2$	     					& 	3	  &	$-2\alpha_1-3\alpha_3$				& 0	\\
$s_3s_2s_1$	     					& 	3	  &	$\alpha_2-5\alpha_3$				& 0	\\
$s_3s_2s_3$	     					& 	3	  &	$\alpha_1-\alpha_2-4\alpha_3$			& 0	\\

$s_1s_2s_3s_1$	     				& 	4	  &	$-3\alpha_1-2\alpha_2+\alpha_3$		& 0	\\
$s_1s_2s_3s_2$	     				& 	4	  &	$-5\alpha_1-3\alpha_2-3\alpha_3$		& 0	\\
$s_2s_3s_1s_2$	     				& 	4	  &	$-2\alpha_1-6\alpha_2-3\alpha_3$		& 0	\\
$s_2s_3s_2s_1$	     				& 	4	  &	$-5\alpha_2-5\alpha_3$				& 0	\\
$s_3s_1s_2s_1$	     				& 	4	  &	$-2\alpha_1-\alpha_2-5\alpha_3$		& 0	\\
$s_3s_1s_2s_3$	     				& 	4	  &	$-3\alpha_1-\alpha_2-4\alpha_3$		& 0	\\
$s_3s_2s_3s_1$	     				& 	4	  &	$-2\alpha_2-6\alpha_3$				& 0	\\
$s_3s_2s_3s_2$	     				& 	4	  &	$\alpha_1-3\alpha_2-4\alpha_3$		& 0	\\

$s_1s_2s_3s_1s_2$	     				& 	5	  &	$-5\alpha_1-6\alpha_2-3\alpha_3$		& 0	\\
$s_1s_2s_3s_2s_1$	     				& 	5	  &	$-6\alpha_1-5\alpha_2-5\alpha_3$		& 0	\\
$s_2s_3s_1s_2s_1$	     				& 	5	  &	$-2\alpha_1-7\alpha_2-5\alpha_3$		& 0	\\
$s_2s_3s_1s_2s_3$	     				& 	5	  &	$-3\alpha_1-7\alpha_2-4\alpha_3$		& 0	\\
$s_3s_1s_2s_3s_1$	     				& 	5	  &	$-3\alpha_1-2\alpha_2-6\alpha_3$		& 0	\\
$s_3s_1s_2s_3s_2$	     				& 	5	  &	$-5\alpha_1-3\alpha_2-4\alpha_3$		& 0	\\
$s_3s_2s_3s_1s_2$	     				& 	5	  &	$-2\alpha_1-6\alpha_2-10\alpha_3$		& 0	\\
$s_3s_2s_3s_2s_1$	     				& 	5	  &	$-5\alpha_2-6\alpha_3$				& 0	\\

$s_1s_2s_3s_1s_2s_1$	     			& 	6	  &	$-6\alpha_1-7\alpha_2-5\alpha_3$		& 0	\\
$s_2s_3s_1s_2s_3s_1$	     			& 	6	  &	$-3\alpha_1-8\alpha_2-6\alpha_3$		& 0	\\
$s_2s_3s_1s_2s_3s_3$	     			& 	6	  &	$-5\alpha_1-7\alpha_2-4\alpha_3$		& 0	\\
$s_3s_1s_2s_3s_1s_2$	     			& 	6	  &	$-5\alpha_1-6\alpha_2-10\alpha_3$		& 0	\\
$s_3s_1s_2s_3s_2s_1$	    			& 	6	  &	$-6\alpha_1-5\alpha_2-6\alpha_3$		& 0	\\
$s_3s_2s_3s_1s_2s_1$	     			& 	6	  &	$-2\alpha_1-7\alpha_2-10\alpha_3$		& 0	\\
$s_3s_2s_3s_1s_2s_3$	     			& 	6	  &	$-3\alpha_1-7\alpha_2-11\alpha_3$		& 0	\\

$s_2s_3s_1s_2s_3s_1s_2$     			& 	7	  &	$-5\alpha_1-10\alpha_2-10\alpha_3$	& 0	\\
$s_2s_3s_1s_2s_3s_2s_1$	     		& 	7	  &	$-6\alpha_1-8\alpha_2-6\alpha_3$		& 0	\\
$s_3s_1s_2s_3s_1s_2s_1$	     		& 	7	  &	$-6\alpha_1-7\alpha_2-10\alpha_3$		& 0	\\
$s_3s_2s_3s_1s_2s_3s_1$		     	& 	7	  &	$-3\alpha_1-8\alpha_2-11\alpha_3$		& 0	\\
$s_3s_2s_3s_1s_2s_3s_2$		     	& 	7	  &	$-5\alpha_1-7\alpha_2-11\alpha_3$		& 0	\\

$s_2s_3s_1s_2s_3s_1s_2s_1$	     		& 	8	  &	$-6\alpha_1-10\alpha_2-10\alpha_3$	& 0	\\
$s_3s_2s_3s_1s_2s_3s_1s_2$		     	& 	8	  &	$-5\alpha_1-10\alpha_2-11\alpha_3$		& 0	\\
$s_3s_2s_3s_1s_2s_3s_2s_1$	     		& 	8	  &	$-6\alpha_1-8\alpha_2-11\alpha_3$		& 0	\\

$s_3s_2s_3s_1s_2s_3s_1s_2s_1$	     	& 	9	  &	$-6\alpha_1-10\alpha_2-11\alpha_3$		& 0	\\

\hline
\end{longtable}
\end{table}

\end{example}

The Weyl group of type $B_r$ is a poset with order given by inclusion of sub-words.
To cut down on the number of elements in $W_r$ that we need to consider, we start by describing the set of Weyl group elements $\sigma$ which are not in $\B_r$. 
Any Weyl group element that is greater than or equal to one of the elements listed below will not be in $\B_r$.
\begin{lemma} \label{Br}
No Weyl group element $\sigma $ containing the following products of simple reflections in its reduced word decomposition is in the Weyl alternation set $\B_r$: 
\[ s_{1} s_2 , \   s_2s_1, \   s_2s_3,  \ s_3s_2,  \text{ and } s_rs_{r-1} ,\]
 \[s_i s_{i+1} s_{i+2},  \ s_{i+2} s_{i+1} s_i, \text{ or  } s_{i+1} s_i s_{i+2}  \text{ where } 3 \leq i \leq r-2, \]
 or any product of four consecutive simple reflections $s_i, s_{i+1}, s_{i+2}, s_{i+3}$ in any order.
\end{lemma}

\begin{proof}A simple calculation shows that
\begin{align*}
s_{1}s_2( 2\hroot+2 \rho)-2\rho &= 2\hroot - 6 \alpha_{1} - 4 \alpha_2=-4\alpha_1+4\alpha_3+\cdots+4\alpha_r,\\
s_2 s_{1}(2\hroot+2\rho)-2\rho&=2\hroot-2\alpha_1-6\alpha_2=-2\alpha_2+4\alpha_3+\cdots+4\alpha_r,\\
s_2s_3(2\hroot+2\rho)-2\rho&=2\hroot-6\alpha_2-2\alpha_3=2\alpha_1-2\alpha_2+2\alpha_3+4\alpha_4+\cdots+4\alpha_r,\\
s_3s_2(2\hroot+2\rho)-2\rho&=2\hroot-4\alpha_2-6\alpha_3=2\alpha_1-2\alpha_3+4\alpha_4+\cdots+4\alpha_r,\mbox{ and}\\
s_rs_{r-1}(2\hroot+2\rho)-2\rho&=2\hroot-2\alpha_{r-1}-6\alpha_r=2\alpha_1+4\alpha_2+\cdots+4\alpha_{r-2}+2\alpha_{r-1}-2\alpha_r.
\end{align*}
Thus, no Weyl group element $\sigma$ containing these products of simple reflections $ s_{1} s_2 , \   s_2s_1, \   s_2s_3,  \ s_3s_2,  \text{ or } s_rs_{r-1} $  in its reduced word decomposition is in the Weyl alternation set $\B_r$.

Lemma~\ref{classic} shows that a Weyl group element $\sigma$ containing a product of simple reflections of the form $s_i s_{i+1} s_{i+2}$, $s_{i+2} s_{i+1} s_i$, or  $s_{i+1} s_i s_{i+2}$ 
or a product of four consecutive simple root reflections $s_i, s_{i+1}, s_{i+2}, s_{i+3}$ is not in $\B_r$. 
\end{proof}

We call the subwords described in Lemma \ref{Br} the \emph{basic forbidden subwords} of $\B_r$.  It is easy to see that the vast majority of elements in $W_r$ contain one of these forbidden subwords.  Thus we have greatly reduced the number of elements we must consider.
Now that we have described which elements of $W_r$ are not in $\B_r$, we turn our attention to the elements $\sigma$ which do not contain a forbidden subword.

The next proposition and its corollary describe the Weyl group elements which are in $\B_r$ as commuting products of short strings of simple transpositions.  
We shall refer to these products of simple transpositions listed in Proposition~\ref{basicwords:B} as the \emph{basic allowable subwords} of Type $B$.   
It is essential to note that by definition, the basic allowable subwords are the largest products of consecutive simple reflections that do not contain a forbidden subword.

\begin{proposition}\label{basicwords:B}
The following elements of $W_r$ are in $\B_r$
\begin{itemize}
\item $(r\geq 2)$: 1, i.e. the identity element of $W_r$
\item $(r\geq 3)$: $s_i$ for any $1\leq i\leq r$
\item $(r\geq 4)$: $s_i s_{i+1}$ for any $3\leq i\leq r-1$
\item $(r\geq 5)$: $s_{i+1} s_i$ for any $3\leq i\leq r-2$
\item $(r\geq 5)$: $s_{i} s_{i+1} s_{i}$ for any $3\leq i\leq r-2$
\item $(r\geq 6)$: $s_{i} s_{i+2} s_{i+1}$ for any $3\leq i\leq r-3$.

\end{itemize}
\end{proposition}

\begin{proof}
Recall $\sigma\in\B_r$ if and only if $\sigma(\hroot+\rho)-\rho$ can be written as a nonnegative integral combination of simple roots. Moreover, since we are only concerned with whether or not the 
coefficients are nonnegative integers we know that $\sigma\in\B_r$ if and only if $\sigma(2\hroot+2\rho)-2\rho$.  Also recall that in the Type $B$ case the highest root is $\hroot=\alpha_1+2\alpha_2+\cdots+2\alpha_r.$
Clearly $1\in\B_r$ since $1(\hroot+\rho)-\rho=\hroot$ which can be written as a sum of simple roots with nonnegative integer coefficients. 

Let $r\geq 3$ and $i\in\{1,3,4,\ldots,r\}$. Then by Lemma~\ref{simpleonhroot}
\[s_i(2\hroot+2\rho)-2\rho=2\hroot+(2\rho-2\alpha_i)-2\rho=2\hroot-2\alpha_i=2\alpha_1+4\alpha_2+\cdots+4\alpha_{i-1}+2\alpha_i+4\alpha_{i+1}+\cdots+4\alpha_r,\]
and when $i=2$ we have that
\[s_2(2\hroot+2\rho)-2\rho=2(\hroot-\alpha_2)+(2\rho-2\alpha_2)-2\rho=2\hroot-4\alpha_2=2\alpha_1+4\alpha_3+\cdots+4\alpha_{r}.\]
Hence $s_i\in\B_r$ for all $1\leq i\leq r$.

Let $r\geq 4$ and let $3\leq i\leq r-1$. Then by Lemmas~\ref{simpleonhroot} and \ref{A_2}
\begin{align*}
s_i s_{i+1}(2\hroot+2\rho)-2\rho&=2\hroot+2\rho-4\alpha_i-2\alpha_{i+1}-2\rho\\
&=2\alpha_1+4\alpha_2+\cdots+4\alpha_{i-1}+2\alpha_{i+1}+4\alpha_{i+2}+\cdots+4\alpha_{r}.
\end{align*}
Hence $s_is_{i+1}\in\B_r$, for all $3\leq i\leq r-1$.

Let $r\geq 5$ and let $3\leq i\leq r-2$. Then by Lemmas~\ref{simpleonhroot} and \ref{A_2}
\begin{align*}
 s_{i+1}s_i(2\hroot+2\rho)-2\rho&=2\hroot+2\rho-2\alpha_{i}-4\alpha_{i+1}-2\rho\\
&=2\alpha_1+4\alpha_2+\cdots+4\alpha_{i+1}+2\alpha_i+4\alpha_{i+2}+\cdots+4\alpha_r.
\end{align*}
Hence $s_{i+1}s_i\in\B_r$, for all $3\leq i\leq r-2$.

Let $r\geq 5$ and let $3\leq i\leq r-2$. Then by Lemmas~\ref{simpleonhroot} and \ref{A_2}
\begin{align*} 
s_{i} s_{i+1} s_{i}(2\hroot+2\rho)-2\rho&=2\hroot-4\alpha_i-4\alpha_{i+1}\\
&=2\alpha_1+4\alpha_2+\cdots+4\alpha_{i-1}+4\alpha_{1+2}+\cdots+4\alpha_r.
\end{align*}
Hence $s_is_{i+1}s_i\in\B_r$, for all $3\leq i\leq r-2$.

Let $r\geq 6$ and let $3\leq i\leq r-3$. Then by Lemmas~\ref{simpleonhroot} and \ref{A_3}
\begin{align*} 
s_{i} s_{i+2} s_{i+1}(2\hroot+2\rho)-2\rho&=2\hroot-4\alpha_i-2\alpha_{i+1}-4\alpha_{i+2}\\
&=2\alpha_1+4\alpha_2+\cdots+4\alpha_{r-1}+2\alpha_{i+1}+4\alpha_{i+3}+\cdots+4\alpha_r.
\end{align*}
Hence $s_i s_{i+2} s_{i+1}\in\B_r$, for all $3\leq i\leq r-3$.
\end{proof}
\begin{corollary}\label{productbasicsubwords:B}
If $\sigma\in W$ can be expressed as a commuting product of basic allowable subwords of Type $B$, then $\sigma\in\B_r$.
\end{corollary}

\begin{proof}This follows from the fact that all basic allowable subwords are in $\B_r$ and since they commute they act on disjoint subsets of the indices in the expression $\hroot+\rho$. Hence $\sigma(\hroot+\rho)-\rho$ will continue to be expressible as a non-negative integral combination of simple roots, and thus this commuting product of basic allowable subwords will again be in $\B_r$. 
\end{proof}

We are now ready to state a complete classification of the set $\B_r$ in terms of basic allowable subwords.  
\begin{theorem}\label{setB}
Let $\sigma\in W_r$. Then $\sigma\in\B_r$ if and only if $\sigma$ is a commuting product of basic allowable subwords of Type $B$.
\end{theorem}
\begin{proof}
They Weyl group $W_r$ is a partially ordered set with order given by inclusion of subwords.  Every element of $W_r$ is either greater than or equal to one of the forbidden subwords described in Lemma \ref{Br}, or it is a commuting product of the basic allowable subwords described in Proposition \ref{basicwords:B}.
\end{proof}

\subsection{Cardinality of $\B_r$}
We will now build $\B_r$ recursively in order to determine the cardinality of this set.
For $r\geq 3$, let $P_r$ denote the subset of $\B_r$ of all elements which do not contain a factor of $s_r$. We define $P_0$ and $P_1$ as the empty set and some simple computations show that $P_2=\{1, \ s_1\}$. 

\begin{lemma}\label{PermutationsCarry}
Let $r\geq 3$. If $\sigma \in P_{r-1}$, then $\sigma \in P_{r}$. 
\end{lemma}

\begin{lemma}\label{TranspositionRequirements} 
Let $r \geq 4$. If $\sigma\in P_{r-2}$, then $\sigma  s_{r-1}\in P_r$.
\end{lemma}

\begin{lemma}\label{3CycleRequirements} 
Let $r \geq 5$. If $\sigma\in P_{r-3}$, then $P_r$ will contain $\sigma  s_{r-2} s_{r-1}$, $\sigma  s_{r - 1} s_{r - 2}$, and $\sigma  s_{r-2} s_{r-1} s_{r-2}$.
\end{lemma}

\begin{lemma}\label{4CycleRequirements} 
Let $r \geq 6$. If $\sigma\in P_{r-4}$, then $\sigma  s_{r-3} s_{r-1} s_{r-2}\in P_r$.
\end{lemma}

\begin{proposition}\label{CardinalityP_r}
The cardinality of the set $P_r$ is given by the following recursive formula:
\[|P_r| = |P_{r-1}| + |P_{r-2}| + 3|P_{r-3}| + |P_{r-4}|, \] where $|P_0| = |P_1| = 0, |P_2| = 2, |P_3| = 3$.
\end{proposition}

\begin{proof}
We know that $P_0$ and $P_1$ are the empty set, hence $|P_0|=|P_1|=0$. By definition of $P_r$, we know that $P_2=\{1,s_1\}$ and $P_3=\{1,s_1,s_2\}$, hence $|P_2|=2$ and $|P_3|=3$. Let $P_j\pi=\{\sigma\pi\ | \sigma\in P_j\}$ for any Weyl group element $\pi$ and any positive integer $j$. Then by Lemmas~\ref{PermutationsCarry} and \ref{TranspositionRequirements} we have that $P_4=P_3\cupdot (P_2s_3)=\{1,s_1,s_2,s_3,s_1 s_3\}$.  Hence $|P_4|=|P_3|+|P_2|+3|P_1|+|P_0|=5$. We now proceed by an induction argument on $r$ and by Lemmas~\ref{PermutationsCarry}-\ref{4CycleRequirements} which imply that for any $k\geq 5$, $P_k$ is the union of pairwise disjoint sets

\[P_{k}=P_{k-1}\cupdot (P_{k-2}s_{k-1})\cupdot (P_{k-3}s_{k-2})\cupdot (P_{k-3}s_{k-2} s_{k-1})\cupdot (P_{k-3}s_{k-1}s_{k-2} )\cupdot (P_{k-4}s_{k-3}s_{k-1}s_{k-2}).\]
Thus
\[|P_{k}|=|P_{k-1}|+|P_{k-2}|+3|P_{k-3}|+|P_{k-4}|.\]
\end{proof}

The first 20 terms of the sequence\footnote{This sequence of integers, \href{http://oeis.org/A232162}{A232162}, was added by the authors to The On-Line Encyclopedia of Integer Sequences (OEIS).} $|P_i|$, beginning with $i = 0$: 
 \[0,0,2,3,5,14,30,62,139,305,660,1444,3158,6887,15037,32842,71698,156538,341799,746273,\ldots\]\\
We now need to count the elements of $\B_r$ which contain a factor of $s_r$. To do so, we note the following:

\begin{lemma}\label{AlternatingRequirements} 
Let $r\geq 3$. If $\sigma\in\B_r$ and $\sigma$ contains a factor of $s_r$, then $\sigma=\pi s_r$ for some $\pi\in P_{r-1}$ or $\sigma=\tau s_{r-1}s_r$ for some $\tau\in P_{r-2}$.
\end{lemma}

\begin{corollary}\label{CardinalityN_r}
For $r\geq 2$, the cardinality of the set $\B_r$ is given by the following recursive formula:
\[|\B_r| = |P_{r}| + |P_{r-1}| + |P_{r-2}|.\] \end{corollary}

\begin{proof}
Let $r\geq 2$. Then by Lemma~\ref{AlternatingRequirements} we know that $\B_r$ is the union of three pairwise disjoint sets. Namely 
\[\B_r=P_r\cupdot(P_{r-1}s_r)\cupdot(P_{r-2}s_{r-1}s_r). \]
Thus
\[|\B_r| = |P_{r}|+|P_{r-1}| + |P_{r-2}|.\]
\end{proof}

The first 20 terms of the sequence\footnote{This sequence of integers, \href{http://oeis.org/A232163}{A232163}, was added by the authors to The On-Line Encyclopedia of Integer Sequences (OEIS).} $|\B_i|$, beginning with $i = 2$: \\
\[2, 5, 10, 22, 49, 106, 231, 506, 1104, 2409, 5262, 11489, 25082, 54766, 119577, 261078, \\570035, 1244610, 2717456, 5933249,\ldots\]

\section{Type $C$}\label{TypeC}

When we consider the Lie algebra of type $C$ and rank $r$ we denote the Weyl alternation set as follows:$$\C_r:= \mathcal A(\tilde{\alpha},0)=\{\sigma\in W:\wp(\sigma(\hroot+\rho)-\rho)>0\}.$$ 

\begin{figure}[h]
\begin{picture}(100,75)(-10,-100)
\thicklines
\thinlines
\color{black}
\multiput(95,-70)(50,0){6}{\circle{5}}
\put(97,-70){\line(1,0){46}}
\put(147,-70){\line(1,0){46}}

\multiput(210,-70)(10,0){3}{\circle{2}}
\put(247,-70){\line(1,0){46}}
\put(297,-69){\line(1,0){46}}
\put(297,-71){\line(1,0){46}}
\put(315,-70){\line(1,1){10}}
\put(315,-70){\line(1,-1){10}}
\put(90,-85){$\alpha_1$}
\put(140,-85){$\alpha_2$}
\put(190,-85){$\alpha_3$}
\put(240,-85){$\alpha_{r-2}$}
\put(290,-85){$\alpha_{r-1}$}
\put(340,-85){$\alpha_{r}$}
\end{picture}
\caption{Dynkin diagram of the root system $C_r$} \label{Figure:C_r}
\end{figure}
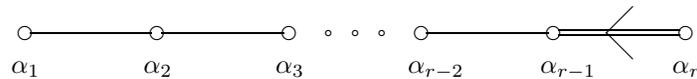

Direct calculations, as those provided in Example~\ref{example}, show that:
\begin{align*}
\C_2&=\{1,s_2\}\\
\C_3&=\{1,s_2,s_3\}\\
\C_4 &= \{1,s_2,s_3,s_4,s_2s_3,s_3s_2,s_2s_3s_2,s_2s_4  \} \\
\C_5 &= \{1,s_2,s_3,s_4,s_5,s_2s_3,s_2s_3s_5,s_3s_2,s_3s_2s_5,s_2s_5,s_2s_3s_2,s_2s_3s_2s_5,s_2s_4,s_3s_4,s_4s_3,s_3s_4s_3,s_2s_4s_3,s_3s_5 \} \\
\C_6 &= \left \{ \begin{matrix} 
1,s_2,s_3,s_4,s_5,s_6,s_2s_3,s_2s_4,s_2s_5,s_2s_6,s_3s_2,s_3s_4,s_3s_5,s_3s_6,s_4s_3,s_4s_5,s_4s_6,\\s_5s_4,s_2s_3s_2,s_2s_3s_5,s_2s_3s_6,s_2s_4s_3, s_2s_4s_5,s_2s_4s_6,s_2s_5s_4,s_3s_2s_5,s_3s_2s_6,s_3s_4s_3, \\s_3s_4s_6,s_3s_5s_4,s_4s_3s_6,s_4s_5s_4, s_2s_3s_2s_5,s_2s_3s_2s_6,s_2s_4s_3s_6,s_2s_4s_5s_4,s_3s_4s_3s_6
\end{matrix} \right \} 
\end{align*}

\begin{remark}
In this case the Weyl group is isomorphic to the group of signed permutations, and hence has order $2^r r!$ where $r$ denotes the rank of the Lie algebra. It is important to note that the cardinalities of the Weyl alternation sets above are much smaller than the order of the respective Weyl group, see Table \ref{Cexample}.
\begin{table}[h]
\centering
\begin{tabular}{|c|c|c|}
\hline
Rank		&	Weyl Alternation Set Cardinality & Weyl Group Order\\
\hline
2		&	2	&8\\
\hline
3		&	3	&48\\
\hline
4		&	8	&384\\
\hline
5		&	18	&3840\\
\hline
6		&	37	&46080\\
\hline
\end{tabular}
\caption{Cardinalities of Weyl alternation sets in comparison to order of Weyl group in Type $C$}
\label{Cexample}
\end{table}
\end{remark}

We now describe the elements $\sigma$ of $W_r$ which are not in the Weyl alternation set $\C_r$ by identifying a list of \emph{forbidden subwords} that prohibit $\sigma$ from being in $\C_r$. 

\begin{lemma} \label{Cr}
Any element $\sigma \in W_r$ containing the following subwords is not in the set $\C_r$:
\[ s_1,  \ s_{r-1}s_r,  \text{ and } s_r s_{r-1},\]\[  s_i s_{i+1} s_{i+2},  \ s_{i+2} s_{i+1} s_i,  \ s_{i+1} s_i s_{i+2}  \text{ where } 1\leq i\leq r-2,\]
or any product of four consecutive simple root reflections $s_i,\ s_{i+1}, \ s_{i+2}, \ s_{i+3},$ in any order.
\end{lemma}
\begin{proof}
Recall that by Lemma~\ref{simpleonhroot}, $s_1(\hroot) = \hroot -2\alpha_1$ and all other simple root reflections fix the highest root, while 
by Lemma~\ref{A_2} we know $s_1(2\rho) = 2\rho-2 \alpha_1.$ So $s_1$ is never in the Weyl alternation set $\C_r$, nor is any word containing $s_1$.
Lemma~\ref{C_2} shows that the Weyl group elements $s_{r-1}s_{r}$ and $s_{r}s_{r-1}$ are never in $\C_r$ nor is any word containing them.
Lemma~\ref{classic} also shows that a Weyl group element $\sigma$ containing a product of simple reflections of the form $s_i s_{i+1} s_{i+2}$, $s_{i+2} s_{i+1} s_i$, or  $s_{i+1} s_i s_{i+2}$ 
where $3 \leq i \leq r-2$ or a product four consecutive simple root reflections $s_i, s_{i+1}, s_{i+2}, s_{i+3}$ is not in the Weyl alternation set $\C_r$.
\end{proof}

We can now describe the elements of the Weyl alternation set $\C_r$ as a product of basic allowable subwords in the following proposition and its corollary.
Each basic allowable subword listed in the following proposition is the largest products of consecutive simple reflections that do not result in a forbidden subword.

\begin{proposition}\label{basicwords:C}
The following elements of $W_r$ are in $\C_r$
\begin{itemize}
\item $(r\geq 2)$: 1, i.e. the identity element of $W_r$
\item $(r\geq 2)$: $s_i$ for any $2\leq i\leq r$
\item $(r\geq 4)$: $s_i s_{i+1}$ for any $2\leq i\leq r-2$
\item $(r\geq 4)$: $s_{i+1} s_i$ for any $2\leq i\leq r-2$
\item $(r\geq 4)$: $s_{i} s_{i+1} s_{i}$ for any $2\leq i\leq r-2$
\item $(r\geq 5)$: $s_{i} s_{i+2} s_{i+1}$ for any $2\leq i\leq r-3$.
\end{itemize}
\end{proposition}

We will refer to the elements listed in Proposition~\ref{basicwords:C} as the \emph{basic allowable subwords} of Type $C$.

\begin{proof}
Recall that in the Type $C$ case the highest root is $\hroot=2\alpha_1+\cdots+2\alpha_{r-1}+\alpha_r$, and that  $\sigma\in\C_r$ if and only if $\sigma(\hroot+\rho)-\rho$ can be written as a nonnegative integral combination of simple roots.  Of course $\sigma(\hroot+\rho)-\rho$ will have nonnegative coefficients if and only if $\sigma(2\hroot+2\rho)-2\rho$ has nonnegative coefficients. We will apply Lemma~\ref{simpleonhroot} and  Lemma~\ref{B_2} in the statement below. 
Clearly $1\in\B_r$ since $1(\hroot+\rho)-\rho=\hroot$ which can be written as a sum of simple roots with nonnegative integer coefficients. 

Let $r\geq 2$ and $2\leq i\leq r$. Then by Lemmas~\ref{simpleonhroot} and \ref{A_2}
\[s_i(2\hroot+2\rho)-2\rho=2\hroot+(2\rho-2\alpha_i)-2\rho=2\hroot-2\alpha_i=4\alpha_1+4\alpha_2+\cdots+4\alpha_{i-1}+2\alpha_i+4\alpha_{i+1}+\cdots+4\alpha_{r-1}+2\alpha_r.\]
Hence $s_i\in\C_r$ for all $2\leq i\leq r$, with $r\geq 2$.

Let $r\geq 4$ and let $2\leq i\leq r-2$. Then by Lemmas~\ref{simpleonhroot} and \ref{A_2}
\[s_i s_{i+1}(2\hroot+2\rho)-2\rho=2\hroot+(2\rho-4\alpha_i-2\alpha_{i+1})-2\rho=2\hroot-4\alpha_i-2\alpha_{i+1}=4\alpha_1+\cdots+4\alpha_{i-1}+2\alpha_{i+1}+4\alpha_{i+2}+\cdots+4\alpha_{r-1}+2\alpha_r.\]
Hence $s_is_{i+1}\in\C_r$, for all $2\leq i\leq r-2$, with $r\geq 4$.

Let $r\geq 4$ and let $2\leq i\leq r-2$. Then by Lemmas~\ref{simpleonhroot} and \ref{A_2}  
\[ s_{i+1}s_i(2\hroot+2\rho)-2\rho=2\hroot+(2\rho-2\alpha_i-4\alpha_{i+1})-2\rho=2\hroot-2\alpha_i-4\alpha_{i+1}=4\alpha_1+\cdots+4\alpha_{i-1}+2\alpha_{i}+4\alpha_{i+2}+\cdots+4\alpha_{r-1}+2\alpha_r.\]
Hence $s_{i+1}s_i\in\C_r$, for all $2\leq i\leq r-2$, with $r\geq 4$.

Let $r\geq 4$ and let $2\leq i\leq r-2$. Then by Lemmas~\ref{simpleonhroot} and \ref{A_2}
\begin{align*} 
s_{i} s_{i+1} s_{i}(2\hroot+2\rho)-2\rho&=2\hroot+(2\rho-4\alpha_i-4\alpha_{i+1})-2\rho\\
&=4\alpha_1+\cdots+4\alpha_{i-1}+4\alpha_{i+2}+\cdots+4\alpha_{r-1}+2\alpha_r.
\end{align*}
Hence $s_is_{i+1}s_i\in\C_r$, for all $2\leq i\leq r-2$, with $r\geq 4$.

Let $r\geq 5$ and let $2\leq i\leq r-3$. Then by Lemmas~\ref{simpleonhroot} and \ref{A_3}
\begin{align*} 
s_{i} s_{i+2} s_{i+1}(2\hroot+2\rho)-2\rho&=2\hroot+(2\rho-4\alpha_i-2\alpha_{i+1}-4\alpha_{i+2})-2\rho\\
&=4\alpha_1+\cdots+4\alpha_{i-1}+2\alpha_{i+1}+4\alpha_{i+3}+\cdots+4\alpha_{r-1}+2\alpha_r.
\end{align*}
Hence $s_i s_{i+2} s_{i+1}\in\C_r$, for all $2\leq i\leq r-3$, with $r\geq 5$.
\end{proof}

\begin{corollary}\label{productbasicsubwords:C}
If $\sigma\in W$ can be expressed as a product of commuting basic allowable subwords of Type $C$, then $\sigma\in\C_r$.
\end{corollary}

\begin{proof}  By Proposition~\ref{basicwords:C}, all basic allowable subwords are in $\C_r$.  Moreover, two basic allowable subwords commute if and only if they act on disjoint sets of simple roots.
Hence, in a product of commuting basic allowable subwords each subword acts on nonconsecutive indices of the expression $\hroot+\rho$. Hence the expression $\sigma(\hroot+\rho)-\rho$ will continue to be expressible as a non-negative integral combination of simple roots, and thus a product of commuting basic allowable subwords will again be in $\C_r$. 
\end{proof}

\begin{theorem}\label{setC}
Let $\sigma\in W_r$. Then $\sigma\in\C_r$ if and only if $\sigma$ is a product of commuting basic allowable subwords of Type $C$.
\end{theorem}
\begin{proof}
The Weyl group $W_r$ is a partially ordered set with order defined by inclusion of subwords.  It is easy to see that an element $\sigma \in W_r$ either contains one of the forbidden subwords 
listed in Lemma \ref{Cr}, or it is a product of commuting basic allowable subwords described in Proposition \ref{basicwords:C} \end{proof}

\subsection{Cardinality of $\C_r$}
We will now build $\C_r$ recursively in order to determine the cardinality of this set. For $r\geq 3$, let $P_r$ denote the subset of $\C_r$ of all elements which do not contain a factor of $s_r$. We define $P_0$ as the empty set and some simple computations show that $P_1=P_2=\{1\}$ and $P_3=\{1,s_2\}$. 

\begin{lemma}\label{PermutationsCarryC}
Let $r\geq 2$. If $\sigma \in P_{r-1}$, then $\sigma \in P_{r}$. 
\end{lemma}

\begin{lemma}\label{TranspositionRequirementsC} 
Let $r \geq 3$. If $\sigma\in P_{r-2}$, then $\sigma  s_{r-1}\in P_r$.
\end{lemma}

\begin{lemma}\label{3CycleRequirementsC} 
Let $r \geq 4$. If $\sigma\in P_{r-3}$, then $P_r$ will contain $\sigma  s_{r-2} s_{r-1}$, $\sigma  s_{r - 1} s_{r - 2}$, and $\sigma  s_{r-2} s_{r-1} s_{r-2}$.
\end{lemma}

\begin{lemma}\label{4CycleRequirementsC} 
Let $r \geq 5$. If $\sigma\in P_{r-4}$, then $\sigma  s_{r-3} s_{r-1} s_{r-2}\in P_r$.
\end{lemma}

\begin{proposition}\label{CardinalityP_rC}
The cardinality of the set $P_r$ is given by the following recursive formula:
\[|P_r| = |P_{r-1}| + |P_{r-2}|+ 3|P_{r-3}|+|P_{r-4}|, \] where $|P_0|  = 0,  |P_1|=|P_2| = 1, |P_3| = 2$.
\end{proposition}

\begin{proof}
We know that $P_0$ is the empty set, hence $|P_0|=0$. By definition of $P_r$ and some basic computations we can show that $P_1=P_2=\{1\}$ and $P_3=\{1,s_2\}$, hence $ |P_1|=|P_2|=1$ and $|P_3|=2$. Let $P_j\pi=\{\sigma\pi\ | \sigma\in P_j\}$ for any Weyl group element $\pi$ and any positive integer $j$. Then by Lemmas~\ref{PermutationsCarryC} and \ref{TranspositionRequirementsC} we have that $P_4=P_3\cupdot (P_2s_3)\cupdot(P_1s_2s_3)\cupdot(P_1s_3s_2)\cupdot(P_1s_2s_3s_2)=\{1,s_2,s_3,s_2s_3,s_3s_2,s_2s_3s_2\}$.  Hence $|P_4|=|P_3|+|P_2|+3 |P_1|+|P_0|=2+1+3(1)+0=6$. We now proceed by an induction argument on $r$ and by Lemmas~\ref{PermutationsCarryC}-\ref{4CycleRequirementsC} which imply that for any $k\geq 5$, $P_k$ is the union of pairwise disjoint sets

\[P_{k}=P_{k-1}\cupdot (P_{k-2}s_{k-1})\cupdot (P_{k-3}s_{k-2}s_{k-1})\cupdot (P_{k-3}s_{k-1} s_{k-2})\cupdot (P_{k-3}s_{k-2}s_{k-1}s_{k-2} )\cupdot (P_{k-4}s_{k-3}s_{k-1}s_{k-2}).\]
Thus
\[|P_{k}|=|P_{k-1}|+|P_{k-2}|+ 3|P_{k-3}|+|P_{k-4}|.\]
\end{proof}
The first 20 terms of the sequence\footnote{This sequence of integers, \href{http://oeis.org/A232164}{A232164}, was added by the authors to The On-Line Encyclopedia of Integer Sequences (OEIS).} $|P_i|$, beginning with $i = 0$: 
 \[0,1,1,2,6,12,25,57,124,268,588,1285,2801,6118,13362,29168,63685,139057,303608,662888,1447352,\ldots\]

We now need to count the elements of $\C_r$ which contain a factor of $s_r$. To do so, we note the following:

\begin{lemma}\label{AlternatingRequirementsC} 
Let $r\geq 2$. If $\sigma\in\C_r$ and $\sigma$ contains a factor of $s_r$, then $\sigma=\pi s_r$ for some $\pi\in P_{r-1}$.
\end{lemma}

\begin{corollary}\label{CardinalityN_rC}
For $r\geq 2$, the cardinality of the set $\C_r$ is given by the following recursive formula:
\[|\C_r| = |P_{r}| + |P_{r-1}|.\] \end{corollary}

\begin{proof}
Let $r\geq 2$. Then by Lemma~\ref{AlternatingRequirementsC} we know that $\C_r$ is the union of two pairwise disjoint sets. Namely 
\[\C_r=P_r\cupdot(P_{r-1}s_r). \]
Thus
\[|\C_r| = |P_{r}|+|P_{r-1}|.\]
\end{proof}

The first 20 terms of the sequence\footnote{This sequence of integers, \href{http://oeis.org/A232165}{A232165}, was added by the authors to The On-Line Encyclopedia of Integer Sequences (OEIS).} $|\C_i|$, beginning with $i = 2$: \\
\[2,3,8,18,37,82,181,392,856,1873,4086,8919,19480,42530,92853,202742,442665,966496,2110240,4607473,\ldots\]

\section{Type $D$}\label{TypeD}

\begin{figure}[h]
\begin{picture}(100,75)(-10,-100)
\thicklines
\thinlines
\color{black}
\multiput(95,-70)(50,0){6}{\circle{5}}
\put(97,-70){\line(1,0){46}}
\put(147,-70){\line(1,0){46}}
\put(295,-68){\line(0,1){25}}
\multiput(210,-70)(10,0){3}{\circle{2}}
\put(247,-70){\line(1,0){46}}
\put(297,-70){\line(1,0){46}}
\put(295,-40){\circle{5}}
\put(90,-85){$\alpha_1$}
\put(140,-85){$\alpha_2$}
\put(190,-85){$\alpha_3$}
\put(240,-85){$\alpha_{r-3}$}
\put(290,-85){$\alpha_{r-2}$}
\put(300,-45){$\alpha_{r-1}$}
\put(340,-85){$\alpha_{r}$}
\end{picture}
\caption{Dynkin diagram of the root system $D_r$} \label{Figure:D_r}
\end{figure}
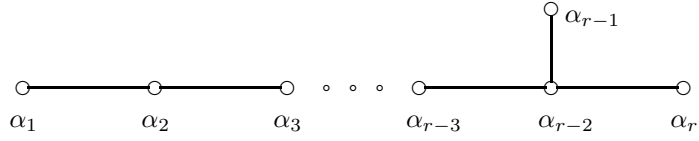

When we consider the Lie algebra of type $D$ and rank $r\geq 4$ we denote the Weyl alternation set as follows:\begin{align}
\D_r:&= \mathcal A(\tilde{\alpha},0)=\{\sigma\in W:\wp(\sigma(\hroot+\rho)-\rho)>0\}.\label{thedset}
\end{align}
Direct calculations, as those provided in Example~\ref{example}, show that:

\begin{align*}
\D_4 &= \{ 1, s_1, s_2, s_3, s_4, s_1s_3, s_1s_4, s_3s_4, s_1s_3s_4 \} \\
\D_5 &= \{ 1, s_1, s_2, s_3, s_4, s_5, s_1s_3, s_1s_4, s_1s_5, s_2s_4,s_2s_5, s_4s_5,s_3s_4, s_3 s_5, s_1s_4s_5,s_1s_3s_4, s_1s_3s_5, s_2s_4s_5\} \\
\D_6 &= \left \{ \begin{matrix}  1,  s_1, s_2, s_3, s_4, s_5, s_6, s_1s_3, s_1s_4, s_1s_5, s_1s_6, s_2s_4, s_2s_5,  s_2s_6, s_3s_4, s_3s_5, s_3s_6 , s_4s_3, s_4s_5, s_4s_6, s_5s_6  ,
s_1s_3s_4, \\s_1s_3s_5, s_1s_3s_6 ,s_1s_4s_3,  s_1s_4s_5, s_1 s_4s_6, s_1s_5s_6, s_2s_4s_5, s_2s_4 s_6, s_2s_5s_6, s_3 s_4 s_3, s_3 s_5 s_6, s_1 s_3 s_4 s_3, s_1s_3s_5 s_6 \end{matrix} \right \}  \\
\D_7 &= 
\left \{ \begin{matrix}  
1, s_1, s_2 , s_3, s_4 ,s_5 ,s_6,  s_7,s_ 1s_3, 
 s_3s_ 4,  
 s_ 1s_4,
 s_ 2s_4, s_4s_ 3,  s_4s_ 5,   s_ 1s_5,   s_ 2s_5,   s_ 3s_5,   s_5s_ 4,   s_5s_ 6,    s_5s_ 7,  s_ 1s_6,   s_ 2s_6,  
 s_ 3s_6,  \\
 s_ 4s_6,  
 s_ 1s_7, 
 s_ 2s_7,  
 s_ 3s_7,  
 s_ 4s_7,  
 s_7s_ 6,  
 s_ 1s_3s_ 4, 
 s_3s_ 4s_ 3,  
 s_ 1s_4s_ 3,  
 s_ 1s_4s_ 5, 
 s_ 2s_4s_ 5,  
 s_4s_ 5s_ 4,  
 s_ 1s_ 3s_5,  
 s_5s_ 3s_ 4,  
 s_ 1s_5s_ 4,  \\
 s_ 2s_5s_ 4,  
 s_ 1s_5s_ 6,  
 s_ 2s_5s_ 6,  
 s_ 3s_5s_ 6,  
 s_ 1s_5s_ 7,  
 s_ 2s_5s_ 7,  
 s_ 3s_5s_ 7, 
 s_ 1s_ 3s_6,  
 s_ 3s_ 4s_6,  
 s_ 1s_ 4s_6,  
 s_ 2s_ 4s_6,  
 s_ 4s_ 3s_6,  
 s_ 1s_ 3s_7,  
 s_ 3s_ 4s_7,  \\
 s_ 1s_ 4s_7, 
 s_ 2s_ 4s_7,  
 s_ 4s_ 3s_7,  
 s_ 1s_7s_ 6,  
 s_ 2s_7s_ 6,  
 s_ 3s_7s_ 6,    
 s_ 4s_7s_ 6,  
 s_ 1s_3s_ 4s_ 3,   
 s_ 1s_4s_ 5s_ 4,  
 s_ 2s_4s_ 5s_ 4, 
 s_ 1s_5s_ 3s_ 4, 
 s_ 1s_ 3s_5s_ 6,\\
 s_ 1s_ 3s_5s_ 7,  
 s_ 1s_ 3s_ 4s_6,  
 s_ 3s_ 4s_ 3s_6,  
 s_ 1s_ 4s_ 3s_6, 
 s_ 1s_ 3s_ 4s_7,  
 s_ 3s_ 4s_ 3s_7,  
 s_ 1s_ 4s_ 3s_7,  
 s_ 1s_ 3s_7s_ 6, 
 s_ 3s_ 4s_7s_ 6,  
 s_ 1s_ 4s_7s_ 6, \\ 
 s_ 2s_ 4s_7s_ 6, 
 s_ 4s_ 3s_7s_ 6,  
 s_ 1s_ 3s_ 4s_ 3s_6, 
 s_ 1s_ 3s_ 4s_ 3s_7,  
 s_ 1s_ 3s_ 4s_7s_ 6, 
 s_ 3s_ 4s_ 3s_7s_ 6,
 s_ 1s_ 4s_ 3 s_7s_ 6,  
 s_1s_3s_4s_3s_6s_7  \end{matrix} \right \} 
\end{align*}

We start by identifying a list of forbidden subwords that are not in $\D_r$.  
\begin{lemma}\label{Dr}
Any Weyl group element $\sigma \in W_r$ containing the following subwords is not in the Weyl alternation set $\D_r$
\[ s_{1} s_2,  \ s_2 s_1,  \ s_2s_3,  \ s_3s_2,  \ s_{r-1}s_{r-2}, \text{or } s_rs_{r-2} \]
\[  s_i s_{i+1} s_{i+2} ,  \ s_{i+2} s_{i+1} s_i ,  \text{ or }  s_{i+1} s_i s_{i+2} \text{ where }1\leq i\leq r-2 \]
In addition, any $\sigma$ containing a product of four consecutive simple reflections $s_i,\ s_{i+1},\ s_{i+2},\ s_{i+3}$ in any order, will not be in $\D_r$.  
\end{lemma}

\begin{proof}
We calculate that $s_{1} s_2$, $s_2 s_1$, $s_2s_3$, $s_3s_2$, $s_{r-1}s_{r-2}$, and $ s_rs_{r-2}$  are not in the Weyl alternation set $\D_r$ because
\begin{align*}
s_{1}s_2( 2\hroot+2 \rho)-2\rho &= 2\hroot - 6 \alpha_{1} - 4 \alpha_2=-4\alpha_1+4\alpha_3+\cdots+4\alpha_{r-2}+2\alpha_{r-1}+2\alpha_r,\\
s_2 s_{1}(2\hroot+2\rho)-2\rho&=2\hroot-2\alpha_1-6\alpha_2=-2\alpha_2+4\alpha_3+\cdots+4\alpha_{r-2}+2\alpha_{r-1}+2\alpha_r,\\
s_2s_3(2\hroot+2\rho)-2\rho&=2\hroot-6\alpha_2-2\alpha_3=2\alpha_1-4\alpha_2+2\alpha_3+4\alpha_4+\cdots+4\alpha_{r-2}+2\alpha_{r-1}+2\alpha_r,\\
s_3s_2(2\hroot+2\rho)-2\rho&=2\hroot-4\alpha_2-6\alpha_3=2\alpha_1-2\alpha_3+4\alpha_4+\cdots+4\alpha_{r-2}+2\alpha_{r-1}+2\alpha_r,\\
s_{r-1}s_{r-2}(2\hroot+2\rho)-2\rho&=2\hroot-2\alpha_{r-2}-4\alpha_{r-1}=2\alpha_1+4\alpha_2+\cdots+4\alpha_{r-2}-2\alpha_r,\mbox{ and}\\
s_rs_{r-2}(2\hroot+2\rho)-2\rho&=2\hroot-2\alpha_{r-2}-4\alpha_r=2\alpha_1+4\alpha_2+\cdots+4\alpha_{r-3}+2\alpha_{r-2}+2\alpha_{r-1}-2\alpha_r.
\end{align*}
Thus $\sigma$ cannot contain any of the above subwords as factors in its reduced word expression.

Now Lemma~\ref{classic} shows that if $\sigma$ contains any of the subwords $s_i s_{i+1} s_{i+2}$, $s_{i+2} s_{i+1} s_i$, or $s_{i+1} s_i s_{i+2}$ with $1\leq i\leq r-2$ or a product of 
four consecutive simple root reflections, then $\sigma$ is not in $\D_r$.  
\end{proof}

We have identified a large set of elements in $W_r$ which are not in the Weyl alternation set $\D_r$.  Now we will show that the remaining elements are in $\D_r$ and describe them
as products of basic allowable subwords as follows.

\begin{proposition}\label{basicwords:D}
The following elements of $W_r$ are in $\D_r$
\begin{itemize}
\item $(r\geq 2)$: 1, i.e. the identity element of $W_r$
\item $(r\geq 3)$: $s_i$ for any $1\leq i\leq r$
\item $(r\geq 4)$: $s_{i} s_{i+1}$ for any $3\leq i\leq r-1$
\item $(r\geq 6)$: $s_{i+1} s_{i}$ for any $3\leq i\leq r-3$ 
\item $(r\geq 6)$: $s_{i} s_{i+1} s_{i}$ for any $3\leq i\leq r-3$
\item $(r\geq 7)$: $s_{i} s_{i+2} s_{i+1}$ for any $3\leq i\leq r-4$.
\end{itemize}
\end{proposition}

We will refer to the elements listed in Proposition~\ref{basicwords:D} as the \emph{basic allowable subwords} of Type $D$.

\begin{proof}
Recall that for $1\leq i\leq r$, $s_i(\alpha_i)=-\alpha_i$. If $1\leq i<j\leq r-1$ with $|i-j|=1$ or if $i=r-2$ and $j=r$, then $s_i(\alpha_j)=s_j(\alpha_i)=\alpha_i+\alpha_j$.
For $i=r-1$ or $i=r$ we have that $s_{r-1}(\alpha_{r})=\alpha_r$ and $s_{r}(\alpha_{r-1})=\alpha_{r-1}.$ The highest root in this case is $\hroot=\alpha_1+ 2 \alpha_2+ \cdots + 2 \alpha_{r-2}+ \alpha_{r-1}+\alpha_r$.

Observe that $\sigma\in\D_r$ if and only if $\sigma(\hroot+\rho)-\rho$ can be written as a nonnegative integral combination of simple roots. Moreover, since we are only concerned with whether or not the coefficients are nonnegative integers we know that $\sigma\in\D_r$ if and only if $\sigma(2\hroot+2\rho)-2\rho$.

Clearly $1\in\D_r$ since $1(\hroot+\rho)-\rho=\hroot$ which can be written as a sum of simple roots with nonnegative integer coefficients.

Let $r\geq 3$ and observe that by Lemma~\ref{simpleonhroot} and Lemma~\ref{A_2}
\begin{align*}
s_1(2\hroot+2\rho)-2\rho&=2\hroot+2\rho-2\alpha_1-2\rho=4\alpha_2+\cdots+4\alpha_{r-2}+2\alpha_{r-1}+2\alpha_r,\\
s_2(2\hroot+2\rho)-2\rho&=2\hroot-2\alpha_2+2\rho-2\alpha_2-2\rho=2\alpha_1+4\alpha_3+\cdots+4\alpha_{r-2}+2\alpha_{r-1}+2\alpha_r,\\
s_{r-1}(2\hroot+2\rho)-2\rho&=2\hroot+2\rho-2\alpha_{r-1}-2\rho=2\alpha_1+4\alpha_+\cdots+4\alpha_{r-2}+2\alpha_r,\\
s_{r}(2\hroot+2\rho)-2\rho&=2\hroot+2\rho-2\alpha_{r}-2\rho=2\alpha_1+4\alpha_2+\cdots+4\alpha_{r-2}+2\alpha_{r-1}.
\end{align*}

Now for $3\leq i\leq r$ we have that by Lemma~\ref{simpleonhroot} and Lemma~\ref{A_2}
\[s_i(2\hroot+2\rho)-2\rho=2\hroot+(2\rho-2\alpha_i)-2\rho=2\hroot-2\alpha_i=2\alpha_1+4\alpha_2+\cdots+4\alpha_{i-1}+2\alpha_i+4\alpha_{i+1}+\cdots+4\alpha_{r-2}+2\alpha_{r-1}+2\alpha_r.\]
Hence $s_i\in\D_r$ for all $1\leq i\leq r$, with $r\geq 3$.

Now let $r\geq 4$ and $3\leq i\leq r-3$. Then by Lemmas~\ref{simpleonhroot} and \ref{A_2}
\begin{align*}
s_i s_{i+1}(2\hroot+2\rho)-2\rho&=2\hroot+(2\rho-4\alpha_i-2\alpha_{i+1})-2\rho\\
&=2\hroot-4\alpha_i-2\alpha_{i+1}\\
&=2\alpha_1+4\alpha_2+\cdots+4\alpha_{i-1}+2\alpha_{i+1}+4\alpha_{i+2}+\cdots+4\alpha_{r-2}+2\alpha_{r-1}+2\alpha_r.
\end{align*}
Similarly,
\begin{align*}
s_{r-2}s_{r-1}(2\hroot+2\rho)-2\rho&=2\hroot+2\rho-4\alpha_{r-2}-2\alpha_{r-1}-2\rho=2\alpha_1+4\alpha_2+\cdots+4\alpha_{r-3}+2\alpha_r,\\
s_{r-1}s_{r}(2\hroot+2\rho)-2\rho&=2\hroot+2\rho-2\alpha_{r-1}-2\alpha_{r}-2\rho=2\alpha_1+4\alpha_2+\cdots+4\alpha_{r-2}.
\end{align*}
Hence $s_is_{i+1}\in\D_r$, for all $3\leq i\leq r-3$, with $r\geq 4$.

Now let $r\geq 6$ and $3\leq i\leq r-3$. Then by Lemmas~\ref{simpleonhroot} and \ref{A_2}  
\begin{align*} 
s_{i+1}s_i(2\hroot+2\rho)-2\rho&=2\hroot+(2\rho-2\alpha_i-4\alpha_{i+1})-2\rho\\
&=2\hroot-2\alpha_i-4\alpha_{i+1}=2\alpha_1+4\alpha_2\cdots+4\alpha_{i-1}+2\alpha_{i}+4\alpha_{i+2}+\cdots+4\alpha_{r-2}+2\alpha_{r-1}+2\alpha_r.
\end{align*}
Hence $s_{i+1}s_i\in\D_r$, for all $3\leq i\leq r-3$, with $r\geq 6$.

Let $r\geq 6$ and let $3\leq i\leq r-3$. Then by Lemmas~\ref{simpleonhroot} and \ref{A_2}
\begin{align*} 
s_{i} s_{i+1} s_{i}(2\hroot+2\rho)-2\rho&=2\hroot+(2\rho-4\alpha_i-4\alpha_{i+1})-2\rho\\
&=2\alpha_1+4\alpha_2+\cdots+4\alpha_{i-1}+4\alpha_{i+2}+\cdots+4\alpha_{r-2}+2\alpha_{r-1}+2\alpha_r.
\end{align*}
Hence $s_is_{i+1}s_i\in\D_r$, for all $3\leq i\leq r-3$, with $r\geq 6$.

Let $r\geq 7$ and let $3\leq i\leq r-4$. Then by Lemmas~\ref{simpleonhroot} and \ref{A_3}
\begin{align*} 
s_{i} s_{i+2} s_{i+1}(2\hroot+2\rho)-2\rho&=2\hroot+(2\rho-4\alpha_i-2\alpha_{i+1}-4\alpha_{i+2})-2\rho\\
&=2\alpha_1+4\alpha_2+\cdots+4\alpha_{i-1}+2\alpha_{i+1}+4\alpha_{i+3}+\cdots+4\alpha_{r-2}+2\alpha_{r-1}+2\alpha_r.
\end{align*}
Hence $s_i s_{i+2} s_{i+1}\in\D_r$, for all $3\leq i\leq r-4$, with $r\geq 7$.
\end{proof}

\begin{corollary}\label{productbasicsubwords:D}
If $\sigma\in W$ can be expressed as a product of commuting basic allowable subwords of Type $D$, then $\sigma\in\D_r$.
\end{corollary}

\begin{proof}This follows from the fact that all basic allowable subwords are in $\D_r$ by Proposition~\ref{basicwords:D}, and since we are assuming these basic allowable subwords commute, these subwords act on disjoint sets of simple roots in expression $\hroot+\rho$. Hence the expression $\sigma(\hroot+\rho)-\rho$ will continue to be expressible as a non-negative integral combination of simple roots, and thus this disjoint product of basic allowable subwords will again be in $\D_r$. 
\end{proof}

\begin{theorem}\label{setD}
Let $\sigma\in W_r$. Then $\sigma\in\D_r$ if and only if $\sigma$ can be expressed as a product of commuting basic allowable subwords of Type $D$.
\end{theorem}
\begin{proof}
Every element of $W_r$ either contains one of the forbidden subwords described in Lemma \ref{Dr} or it is the product of commuting basic allowable subwords.  
\end{proof}

\subsection{Cardinality of $\D_r$}
To help us recursively count the elements in $\D_r$, we start by defining some special subsets of the support. 
 Letting  $\D_r:=\A_r(\hroot,0)$, as denoted in Equation (\ref{thedset}), we then let $M_r \subset \D_r$ denote the subset of $\D_r$ consisting of 
elements that do not contain $s_1$ in any reduced word decomposition. Let $N_r \subset \D_r$ denote the subset of $\D_r$ consisting of elements that contain $s_1$.  
By definition $N_r = \D_r \setminus M_r$, $\D_r=M_r\cupdot N_r$ and hence $|\D_r| = |M_r|+|N_r|$. 
Let $L_{r} \subset \D_r$ denote the subset of $\D_r$ consisting of elements that do not contain $s_1$ or $s_2$. Note that if $\sigma\in N_r$, 
then there exists $\tau\in L_{r}$ such that $s_1\tau=\sigma$. Hence $|N_r| = |L_{r}|$. 

With this notation in place, we define a map \[ \phi: \D_{r-1} \rightarrow M_r \subset \D_r \] which sends $s_i$ to $s_{i+1}$ for every simple transposition $s_1, \ldots, s_{r-1}$.

We can now characterize the elements of the set $N_r$.
When $r \geq 8$ the elements of $N_r$ are obtained from the sets $L_{r-1}$,$L_{r-2},$ $L_{r-3}$, and $L_{r-4}$
 by either multiplying $s_1$ times a word from $\phi (L_{r-1})$, multiplying $s_1s_3$ times a word from  $\phi^2(L_{r-2})$, multiplying $s_1s_3s_4$, $s_1s_4s_3$, or $s_1s_3s_4s_3$ times a word from $\phi^{3}(L_{r-3})$, or multiplying $s_1s_3s_5s_4$ times a word from $\phi^{4}(L_{r-4}).$

Since $|N_r| = |L_r|$ this implies that the cardinality of $N_r$ satisfies the following recursion:
 \begin{equation}  |N_r| =  |N_{r-1}| + |N_{r-2}|+3|N_{r-3}|+ |N_{r-4}| . \label{N_r} \end{equation} 
Next we characterize the elements of the set $M_r$. 
Every element of $M_r$ either contains $s_2$ or it does not.  The ones that contain $s_2$ are obtained by multiplying $s_2$ times the elements of $\phi(L_{r-1})$. 
The elements of $M_r$ by definition do not contain $s_1$, so if in addition they do not contain $s_2$ they are, again by definition, all elements of $L_r$.
This implies that $|M_r|$ satisfies the following recursion:
\begin{equation} |M_r| = |L_r|+|L_{r-1}| = |N_r|+|N_{r-1}| = 2|N_{r-1}| + |N_{r-2}|+3|N_{r-3}|+ |N_{r-4}| . \label{M_r}  \end{equation}
Finally, by the definitions of $\D_r$, $M_r$, and $N_r$ we see that 
 \begin{equation} |\D_r| = |M_r|+ | N_r| = 2|N_r|+|N_{r-1}| = 3|N_{r-1}|+ 2|N_{r+2}| + 6|N_{r-3}| +2|N_{r-4}| . \label{D_r} \end{equation} 
We have listed the elements of $D_r$ for $r \leq 7$ in the previous section.  From these sets, and the recursions described in Equations (\ref{N_r}), (\ref{M_r}), (\ref{D_r}),
 we can find the cardinalities of the sets $\D_r$, $M_r$, $N_r$, and $L_{r}$ for $r \geq 4$.\footnote{These sequences of integers, \href{http://oeis.org/A234576}{A234576}, \href{http://oeis.org/A234597}{A234597}, \href{http://oeis.org/A234599}{A234599}, were added by the authors to The On-Line Encyclopedia of Integer Sequences (OEIS).}  
The cardinalities of the first 16 sets are listed below:

\begin{table}[h]
\centering
 \begin{tabular}{|c|c|c|c|}   \hline
  $r$   &    $|\D_r|$& $|M_r|$ & $|N_r|$= $|L_{r}|$   \\ \hline 
  4  &  9       &  5     &     4            \\ \hline
 5  &   18      &   11   &    7             \\ \hline
 6  &   35      &   21   &    14             \\ \hline
 7  &   82      &   48   &    34             \\ \hline
 8  &   180     &  107   &    73             \\ \hline
 9  &   385     &   229  &    156            \\ \hline
 10  &  846     &   501  &    345             \\ \hline
 11  &  1853    &  1099  &    754             \\ \hline
 12  &  4034    &  2394  &    1640             \\ \hline
 13  &  8810    &  5225  &    3585             \\ \hline  
 14  &  19249   &  11417 &    7832             \\ \hline  
 15  &  42014   &  24923 &    17091             \\ \hline 
 16  &  91727   &  54409 &    37318             \\ \hline 
 17  &  200298  &  118808&    81490             \\ \hline 
 18  &  437316  &  259403&    177913             \\ \hline 
 19  &  954809  &  566361&    388448             \\ \hline 

 \end{tabular}

\end{table}

\color{black}

\section{$q$-analog of Kostant's weight multiplicity formula}\label{qanalog}

The $q$-analog of Kostant's weight multiplicity ($q$-multiplicity) is defined as follows

\begin{align}
m_q(\lambda,\mu)&=\sum_{\sigma\in W}(-1)^{\ell(\sigma)}\wp_q(\sigma(\lambda+\rho)-\rho-\mu),\label{eq}
\end{align}
where $\wp_q$ denotes the $q$-analog of Kostant's partition function. That is, for any weight $\xi$, we have that 
\[\wp_q(\xi)=c_0+c_1q+c_2q^2+c_3q^3+\cdots+c_kq^k,\]
where $c_i$ is the number of ways to write $\xi$ as a sum of exactly $i$ positive roots.

Given $\mathfrak{g}$, a simple Lie algebra of rank $r$ with highest root $\hroot$, it is known that $m_q(\hroot,0)=\sum_{i=1}^r q^{e_i}$, where $e_1,\cdots,e_r$ are the exponents of $\mathfrak{g}$,~\cite{Kostant}. Using the $q$-analog values of the partition function as listed in Tables~\ref{table1B}-\ref{table2D} we can prove that for the Lie algebras of Type $B$ (ranks 2, 3, 4, 5, and 6), Type $C$ (ranks 2, 3, 4, 5, and 6) and Type $D$ (ranks 4, 5, and 6):
 \[m_q(\tilde\alpha,0)=\sum_{\sigma\in W}(-1)^{\ell(\sigma)}\wp_{q}(\sigma(\tilde\alpha+\rho)-\rho)=q^{e_1}+q^{e_2}+\cdots+q^{e_r},\]
where $e_1, e_2,\ldots, e_r$ are given in Table~\ref{exps}.
\begin{table}[htp]
\centering
\begin{tabular}{|l|l|}
\hline
Lie algebra & Exponents\\
\hline
$A_r$ & $1, 2, 3, \ldots, r$\\
$B_r$ & $1, 3, 5, \ldots, 2r-1$\\
$C_r$ & $1, 3, 5, \ldots, 2r-1$\\
$D_r$ & $1, 3, 5, \ldots, 2r-3,r-1$\\
\hline

\end{tabular}
\caption{Exponents of simple Lie algebras.}
\label{exps}
\end{table}

Notice that in Tables~\ref{table1B}-\ref{table2D}, the column labeled \emph{string} denotes the coefficients of the simple roots for the expression $\sigma(\hroot+\rho)-\rho.$ For example, in Table~\ref{table2B} the \emph{string} of 102201, corresponding to the Weyl group element $s_2s_5s_6$, is short hand notation for the fact that $s_2s_5s_6(\hroot+\rho)-\rho=\alpha_1+2\alpha_3+2\alpha_4+\alpha_6.$

A complete combinatorial proof of the above result for Lie algebras of type $B$, $C$, and $D$ is yet to be completed. Such a proof would require a closed formula for the $q$-analog of Kostant's partition function, $\wp_q(\sigma(\hroot+\rho)-\rho)$, for every element in the respective Weyl alternation set, see Section \ref{open}. In \cite{PH}, Harris provided the necessary results to complete such a proof in the Type $A$ case.

\section{Non-zero Weight Spaces}
It is fundamental in Lie theory that the zero weight space is a Cartan subalgebra, and that the non-zero weights of $L(\hroot)$, the adjoint representation of $\mathfrak{g}$, are the roots and have multiplicity 1. We visit this from our point of view in the case of the Lie algebras of Types $B$, $C$, and $D$. First we begin with the following general result.

\begin{theorem}\label{mu=lambda}
Let $\lambda$ be a dominant integral weight of the simple Lie algebra $\mathfrak{g}$ of rank $r$. Then $\sigma(\lambda+\rho)-\lambda-\rho$ can be written as a nonnegative integral sum of positive roots if and only if $\sigma$ is the identity.
\end{theorem}
\begin{proof}
($\Rightarrow$) If $\sigma\neq 1$, then there exists nonnegative integers $m_1,\ldots,m_j$ between 1 and $r$, such that $\sigma(\lambda+\rho)=\lambda+\rho-\sum_{i=1}^j m_i\alpha_i$. Then $\sigma(\lambda+\rho)-\lambda-\rho=-\sum_{i=1}^j m_i\alpha_i$. Hence $\sigma(\lambda+\rho)-\lambda-\rho$ cannot be written as nonnegative integral sum of positive roots.

($\Leftarrow$) If $\sigma=1$, then $\sigma(\lambda+\rho)-\lambda-\rho=0$, which can be written as a nonnegative integral combination of positive roots as desired.
\end{proof}

Recall that the fundamental weights (relative to the choice of simple roots) are the elements $\varpi_1,\ldots,\varpi_r$ of $\mathfrak{h}^*$ which are dual to the coroot basis $\{\check{\alpha_1},\ldots,\check{\alpha_r}\}$, see \cite{GW} for notation. Also recall that in every Lie type the highest root is a dominant weight since it is the highest weight of the adjoint representation. It is a simple exercise\footnote{See exercise 3.2.5 \#1(a) in \cite{GW}.} to show that the only dominant positive roots are:
\begin{itemize} 
\item Type $A_r$: $\hroot=\varpi_1+\varpi_r$,
\item Type $B_r$: $\hroot=\varpi_2$ and $\varpi_1=\alpha_1+\cdots+\alpha_r,$ 
\item Type $C_r$: $\hroot=2\varpi_1$ and $\varpi_2=\alpha_1+2\alpha_2+\cdots+2\alpha_{r-1}+\alpha_r,$
\item Type $D_r$: $\hroot=\varpi_2$.
\end{itemize}

Since in all Lie types the highest root is dominant Theorem~\ref{mu=lambda} implies the following.
 \begin{corollary}\label{set=1}
Let $\hroot$ denote the highest root of the Lie algebra of Type $A$, $B$, $C$, or $D$, respectively. Then, in each respective Lie type, the Weyl alternation set associated to the pair of dominant weights $\lambda=\hroot$ and $\mu=\hroot$  is given by $\A(\hroot,\hroot)=\{1\}.$
\end{corollary}

Recall that given $\mu\in P(\mathfrak{g})$, there exists $w\in W$ and $\xi\in P_+(\mathfrak{g})$ such that $w(\xi) =\mu$ and  given that weight multiplicities are invariant under $W$ (Propositions 3.1.20, 3.2.27 in \cite{GW}) it suffices to consider $\mu\in P_+(\mathfrak{g})$. Thus Corollary~\ref{set=1} implies that for all Lie types, $m(\hroot,\mu)=1$, whenever $\mu\in \Phi$.

However, it is interesting to consider the remaining cases where there exists a dominant positive root, which is not the highest root. Namely the case $\lambda=\hroot$ and $\mu=\varpi_1$ in Type $B$ and the case $\lambda=\hroot$ and $\mu=\varpi_2$ in Type $C$. 
\begin{theorem}\label{omega1set}
Let $\sigma\in W$, then $\sigma\in\B_r(\hroot,\varpi_1)$ if and only if $\sigma=1$ or $\sigma=s_{i_1}s_{i_2}\cdots s_{i_j}$, where $i_1,\ldots,i_j$ are non-consecutive integers between $3$ and $r$.
\end{theorem}

\begin{proof}
Recall $\sigma\in\B_r(\hroot,\varpi_1)$ if and only if $\sigma(\hroot+\rho)-\rho-\varpi_1$ can be written as a nonnegative integral combination of simple roots. Moreover, since we are only concerned with whether or not the coefficients are nonnegative integers we know that $\sigma\in\B_r(\hroot,\varpi_1)$ 
if and only if $\sigma(2\hroot+2\rho)-2\rho-2\varpi_1$ is an equivalent statement. 
Also recall that in the Type $B$ case the highest root is $\hroot=\alpha_1+2\alpha_2+\cdots+2\alpha_r$ and $\varpi_1=\alpha_1+\cdots+\alpha_r.$

$(\Leftarrow):$ Observe that $1(\hroot+\rho)-\rho-\varpi_1=(\alpha_1+2\alpha_2+\cdots+2\alpha_r)-(\alpha_1+\cdots+\alpha_r)=\alpha_2+\cdots+\alpha_r$, which can be written as a sum of simple roots with nonnegative integer coefficients. Thus, if $\sigma=1$, then $\sigma\in\B_r(\hroot,\varpi_1)$. 
Now observe that if $3\leq i\leq r$, then by Lemmas~\ref{simpleonhroot} and \ref{A_2}
\begin{align*}
s_i(2\hroot+2\rho)-2\rho-2\varpi&=2\hroot+(2\rho-2\alpha_i)-2\rho-2\varpi_1\\
&=2\hroot-2\alpha_i-2\varpi_1\\
&=(2\alpha_1+4\alpha_2+\cdots+4\alpha_r)-2\alpha_i-2(\alpha_1+\cdots+\alpha_r)\\
&=2\alpha_2+\cdots+2\alpha_{i-1}+2\alpha_{i+1}+\cdots+2\alpha_r.
\end{align*}
Hence $s_i\in\B_r(\hroot,\varpi_1)$ for all $3\leq i\leq r$.
Suppose $\sigma=s_{i_1}s_{i_2}\cdots s_{i_j}$, where $i_1,\ldots,i_j$ are non-consecutive integers between $3$ and $r$. Then by Lemmas~\ref{simpleonhroot} and \ref{A_2} we have that
\begin{align*}
s_{i_1}s_{i_2}\cdots s_{i_j}(2\hroot+2\rho)-2\rho-2\varpi_1&=2\hroot+2\rho-2(\alpha_{i_1}+\alpha_{i_2}+\cdots+\alpha_{i_j})-2\rho-2\varpi_1\\
&=(2\alpha_1+4\alpha_2+\cdots+4\alpha_r)-2(\alpha_{i_1}+\alpha_{i_2}+\cdots+\alpha_{i_j})-2(\alpha_1+\cdots+\alpha_r)\\
&=(2\alpha_2+\cdots+2\alpha_r)-2(\alpha_{i_1}+\alpha_{i_2}+\cdots+\alpha_{i_j}).
\end{align*}
Thus $\sigma\in\B_r(\hroot,\varpi_1)$ as claimed.

$(\Rightarrow):$ Suppose that $\sigma\in\B_r(\hroot,\varpi_1)$. If $\sigma=1$, we are done. So suppose that $\sigma$ is not the identity element. First notice that 
\begin{align*}
s_1(2\hroot+2\rho)-2\rho-2\varpi_1&=2\hroot+(2\rho-2\alpha_1)-2\rho-2\varpi_1\\
&=(2\alpha_1+4\alpha_2+\cdots+4\alpha_r)-2\alpha_1-2(\alpha_1+\cdots+\alpha_r)\\
&=-2\alpha_1+2\alpha_2+\cdots+2\alpha_r
\end{align*}
and
\begin{align*}
s_2(2\hroot+2\rho)-2\rho-2\varpi_1&=(2\hroot-2\alpha_2)+(2\rho-2\alpha_2)-2\rho-2\varpi_1\\
&=(2\alpha_1+4\alpha_2+\cdots+4\alpha_r)-4\alpha_2-2(\alpha_1+\cdots+\alpha_r)\\
&=-2\alpha_2+2\alpha_3+\cdots+2\alpha_r.
\end{align*}
Hence $\sigma$ cannot contain $s_1$ and $s_2$ as a factor.
Now notice that by Lemmas~\ref{simpleonhroot} and \ref{A_2}
\begin{align*}
s_is_{i+1}(2\hroot+2\rho)-2\rho-2\varpi_1&=2\hroot+(2\rho-4\alpha_i-2\alpha_{i+1})-2\rho-2\varpi_1\\
&=(2\alpha_1+4\alpha_2+\cdots+4\alpha_r)-4\alpha_i-2\alpha_{i+1}-2(\alpha_1+\cdots+\alpha_r)\\
&=2\alpha_2+\cdots+2\alpha_{i-1}-2\alpha_i+2\alpha_{i+2}+\cdots+2\alpha_r
\end{align*}
and
\begin{align*}
s_{i+1}s_i(2\hroot+2\rho)-2\rho-2\varpi_1&=2\hroot+(2\rho-2\alpha_i-4\alpha_{i+1})-2\rho-2\varpi_1\\
&=(2\alpha_1+4\alpha_2+\cdots+4\alpha_r)-2\alpha_i-4\alpha_{i+1}-2(\alpha_1+\cdots+\alpha_r)\\
&=2\alpha_2+\cdots+2\alpha_{i-1}-2\alpha_{i+1}+2\alpha_{i+2}+\cdots+2\alpha_r.
\end{align*}
Therefore $\sigma$ cannot contain any consecutive factors, as claimed. 
\end{proof}
The Fibonacci numbers, denoted by $F_n$ and defined in \cite{Sigler}, are given by the recurrence relation
\[F_n=F_{n-1}+F_{n-2},\] where $F_1=F_2=1$. 
\begin{corollary}\label{fib}
Let $r\geq 2$. Then $|\B_r(\hroot,\varpi_1)|=F_r$.
\end{corollary}
The proof of Corollary \ref{fib} follows from the fact that the $r^{th}$ Fibonacci number, $F_r$, counts the number of ways to choose nonconsecutive integers from the numbers $3, 4, \cdots, r$. Moreover, the following lemmas and propositions follow from analogous arguments as for Lemma 3.1 and Proposition 3.2 in \cite{PH}.

\begin{lemma}\label{nums}
\[|\{\sigma\in\B_r(\hroot,\varpi_1):\ell(\sigma)=k\mbox{ and $\sigma$ contains no $s_r$ factor}\}|=\binom{r-3-k}{k}\]

\[|\{\sigma\in\B_r(\hroot,\varpi_1):\ell(\sigma)=k+1\mbox{ and $\sigma$ contains an $s_r$ factor}\}|=\binom{r-4-k}{k}\]

\[\max  \{\ell(\sigma):\sigma\in\B_r(\hroot,\varpi_1)\mbox{ and $\sigma$ contains no $s_r$ factor}\}  =\left\lfloor \frac{r-3}{2}\right\rfloor\]

\[\max  \{\ell(\sigma):\sigma\in\B_r(\hroot,\varpi_1)\mbox{ and $\sigma$ contains an $s_r$ factor}\}  =\left\lfloor \frac{r-2}{2}\right\rfloor\]
\end{lemma}

We can also compute the value of the $q$-analog of Kostant's partition function, as defined in Section \ref{qanalog}, for every $\sigma\in\B_r(\hroot,\varpi_1)$.

\begin{proposition}\label{value} If $\sigma\in\B_r(\hroot,\varpi_1)$, then 
$$\wp_q(\sigma(\hroot+\rho)-\rho-\varpi_1)=\begin{cases}
q^{1+\ell(\sigma)}(1+q)^{r-1-2\ell(\sigma)}&\mbox{if $\sigma$ contains no $s_r$ factor}\\
q^{\ell(\sigma)} (1+q)^{r-2\ell(\sigma)}&\mbox{if $\sigma$ contains an $s_r$ factor}.
\end{cases}$$ 
 \end{proposition}

\begin{corollary}\label{c1} If $\sigma\in\B_r(\hroot,\varpi_1)$, then $\wp(\sigma(\hroot+\rho)-\rho-\varpi_1)=\begin{cases}2^{r-1-2\ell(\sigma)}&\mbox{if $\sigma$ 
contains no $s_r$ factor}\\ 2^{r-2\ell(\sigma)}&\mbox{if $\sigma$ contains an $s_r$ factor}.\end{cases}$ 
 \end{corollary}
 
 Then we can prove the following result regarding the $q$-multiplicity, Equation \ref{eq}, of the weight $\varpi_1$ in the adjoint representation of the Lie algebra of Type $B$.
 
\begin{theorem}\label{qmult:B} In Type $B$, $m_q(\hroot,\varpi_1)=q^r.$
\end{theorem}

Notice that Corollary~\ref{c1} follows from the fact that $\wp=\wp_{q}|_{q=1}$. This same fact along with Theorem~\ref{qmult:B} implies that the multiplicity 
of the weight $\varpi_1$ in the adjoint representation of the Lie algebra of Type $B$ is 1, as we expected.

 \begin{proof}[Proof of Theorem \ref{qmult:B}]
By Theorem \ref{omega1set} we have that
 \[m_q(\hroot,\varpi)=\displaystyle\sum_{\substack{\sigma\in\B_r(\hroot,\varpi_1)\\ \mbox{with no $s_r$ factor}}}\hspace{-5mm}(-1)^{\ell(\sigma)}\wp_q(\sigma(\hroot+\rho)-\rho-\varpi_1)+\hspace{-5mm}\displaystyle\sum_{\substack{\sigma\in\B_r(\hroot,\varpi_1)\\ \mbox{with an $s_r$ factor}}}\hspace{-5mm}(-1)^{\ell(\sigma)}\wp_q(\sigma(\hroot+\rho)-\rho-\varpi_1).\]

By Lemma \ref{nums} and Proposition \ref{value} we can compute the sums as follows
\[\displaystyle\sum_{\substack{\sigma\in\B_r(\hroot,\varpi_1)\\ \mbox{with no $s_r$ factor}}}\hspace{-5mm}(-1)^{\ell(\sigma)}\wp_q(\sigma(\hroot+\rho)-\rho-\varpi_1)=\displaystyle\sum_{k=0}^{\left\lfloor \frac{r-3}{2}\right\rfloor}(-1)^{k}\binom{r-3-k}{k}q^{1+k}(1+q)^{r-1-2k}\]
and
\[\displaystyle\sum_{\substack{\sigma\in\B_r(\hroot,\varpi_1)\\ \mbox{with an $s_r$ factor}}}\hspace{-5mm}(-1)^{\ell(\sigma)}\wp_q(\sigma(\hroot+\rho)-\rho-\varpi_1)=\displaystyle\sum_{k=0}^{\left\lfloor \frac{r-4}{2}\right\rfloor}(-1)^{1+k}\binom{r-4-k}{k}q^{1+k}(1+q)^{r-2-2k}.\]
By Proposition 3.3 in \cite{PH}, observe that
$$\displaystyle\sum_{k=0}^{\left\lfloor \frac{r-3}{2}\right\rfloor}(-1)^{k}\binom{r-3-k}{k}q^{1+k}(1+q)^{r-1-2k}=q^2\displaystyle\sum_{i=1}^{r-2}q^i$$
and $$\displaystyle\sum_{k=0}^{\left\lfloor \frac{r-4}{2}\right\rfloor}(-1)^{1+k}\binom{r-4-k}{k}q^{1+k}(1+q)^{r-2-2k}=-q^2\displaystyle\sum_{i=1}^{r-3}q^i.$$
Therefore
 \[m_q(\hroot,\varpi)=q^2\displaystyle\sum_{i=1}^{r-2}q^i-q^2\displaystyle\sum_{i=1}^{r-3}q^i=q^r.\]
 \end{proof}

\color{black}
Now we consider the case $\lambda=\hroot$ and $\mu=\varpi_2$ in the Lie algebra of Type $C$.

\begin{theorem}\label{omega2set}
Let $\sigma\in W$. Then $\sigma\in\C(\hroot,\varpi_2)$ if and only if $\sigma=1$.
\end{theorem}

\begin{proof}

Recall that in the Type $C_r$ case the highest root is $\hroot=2\alpha_1+\cdots+2\alpha_{r-1}+\alpha_r$ and $\varpi_2=\alpha_1+2\alpha_2+\cdots+2\alpha_{r-1}+\alpha_r=\hroot-\alpha_1.$

$(\Rightarrow):$ Let $\sigma\in\C_r(\hroot,\varpi_2)$. If $\sigma=1$, then we are done. So suppose $\sigma$ is not the identity. Now observe that by Lemmas~\ref{simpleonhroot} and \ref{A_2}
\[s_1(2\hroot+2\rho)-2\rho-2\varpi_2=(2\hroot-4\alpha_1)+(2\rho-2\alpha_1)-2\rho-2(\hroot-\alpha_1)=-4\alpha_1\]
and for any $2\leq i \leq r$ we have that
\[s_i(2\hroot+2\rho)-2\rho-2\varpi_2=2\hroot+(2\rho-2\alpha_i)-2\rho-2(\hroot-\alpha_1)=2\alpha_1-2\alpha_i.\]
So $\sigma$ cannot contain any factors $s_1,\;\ldots,\;s_r$. Thus $\sigma$ must be the identity. 

$(\Leftarrow):$ Observe that $1(\hroot+\rho)-\rho-\varpi_2=\hroot+\rho-\rho-(\hroot-\alpha_1)=\alpha_1$, hence $1\in\C_r(\hroot,\varpi_2)$.

\end{proof}

\begin{corollary}\label{mult:C} In Type $C$, $m_q(\hroot,\varpi_2)=q$.
\end{corollary}
This follows directly from Theorem~\ref{omega2set} which implies that $m_q(\hroot,\varpi_2)=\wp_q(1(\hroot+\rho)-\rho-\varpi_2)=\wp_q(\alpha_1)=q$. Thus, by setting $q=1$, we have that the multiplicity of the weight $\varpi_2$ in the adjoint representation of the Lie algebra of Type $C$ is 1.

\section{A set of open problems}\label{open}

As stated before computing the values of Kostant's partition function is very difficult. In order to provide a combinatorial proof of the result of Kostant regarding the exponents of the classical Lie algebras of type $B$, $C$ and $D$ one needs to compute the value of the $q-$analog of Kostant's partition function. This is a non-trivial matter. We thus ask the following questions.

\begin{question}\label{onhroot}For any Lie Type: Let $\hroot$ denote the highest root. Can a closed formula for the value of the partition function on the highest root be given? Namely, is there a closed formula for the value of $\wp(\hroot)$? Moreover, can a closed formula for the $q$-analog be given, i.e. $\wp_q(\hroot)$?
\end{question}

The answer to the Type $A$ case is found by setting $q=1$ in Proposition 3.2, in \cite{PH}. Thus, if $\sigma\in\A_r(\hroot,0)$, then $\wp(\sigma(\hroot+\rho)-\rho)=2^{r-1-2\ell(\sigma)}$. We believe that a generating function can be found for the values of $\wp(\hroot)$ at any rank when $\hroot$ is the highest root of the Lie algebras of Type $B$ and $C$. The generating functions associated to the sequences \[1, 3, 11, 40, 145, 525, 1900, 6875, \ldots\] and \[1, 3, 10, 35, 125, 450, 1625, 5875, .\ldots,\] were given in~\cite{juggle} as the number of multiplex juggling sequences\footnote{OEIS sequence \href{http://oeis.org/A136775}{A136775}.} of length $n$, base state $<1,1>$ and hand capacity 2, and the number of periodic multiplex juggling sequences\footnote{OEIS sequence \href{http://oeis.org/A081567}{A081567}.} of length $n$ with base state $<2>$, respectively. We list these generating functions in Table~\ref{genfunc}.

\begin{table}[h!]
\centering
\begin{tabular}{|c|c|}
\hline
Type & Generating function for $\wp(\hroot)$\\
\hline
$B$ &$\frac{x-2x^2+x^3}{1-5x+5x^2}$\\
\hline
$C$&$\frac{x-2x^2}{1-5x+5x^2}$\\
\hline
\end{tabular} 
\caption{Generating functions for the value of $\wp(\hroot)$}
\label{genfunc}
\end{table}
 This would not be the first time that the mathematics of juggling provide insight into such computations. For example, Ehrenborg and Readdy used 
juggling patterns as an application to compute $q$-analogs. In particular they used juggling patterns to compute the Poincar\'{e} series of the affine Weyl group $\tilde{A}_{d-1}$, see \cite{Readdy}.

\begin{problem}
Verify that the generating functions given in Table~\ref{genfunc} are the actual generating functions for the values of Kostant's partition function on the highest root for Lie algebras of Type $B$ and $C$, respectively.
\end{problem}

Other questions we pose deal with the Weyl alternations sets as described in this paper.

\begin{question}\label{Q1} For each positive number $k$, what are the cardinalities of the sets:
\begin{align*} 
\B_{r}^k&=\{\sigma\in\B_r(\hroot,0): \ell(\sigma)=k\},\\
\C_{r}^k&=\{\sigma\in\C_r(\hroot,0): \ell(\sigma)=k\},\\
\D_{r}^k&=\{\sigma\in\D_r(\hroot,0): \ell(\sigma)=k\}?
\end{align*}
\end{question}

\begin{question}\label{Q2} If $\sigma$ is an element of the Weyl alternation set $\B(\hroot,0)$, $\C(\hroot,0)$, or $\D(\hroot,0)$ as computed in Theorems \ref{setB}, \ref{setC}, or \ref{setD}, respectively, then can a closed formula be provided for the value of the $q$-analog of Kostant's partition function: $\wp_q(\sigma(\hroot+\rho)-\rho)$? 
\end{question}
Notice that an answer to Question~\ref{Q2} would immediately provide the values of Kostant's partition function on $\sigma(\hroot+\rho)-\rho$ for any element $\sigma$ of a Weyl alternation set, since the evaluation of $\wp_q$ at $q=1$ recovers the original partition function values. \\ 

By answering Questions~\ref{Q1} and \ref{Q2} one can provide a purely combinatorial proof of Kostant's result regarding the exponents of the respective Lie algebra. That is, if $\mathfrak{g}$ is a classical Lie algebra of rank $r$ with highest root $\hroot$, an answer to the previous questions will yield a purely combinatorial proof that $$m_q(\hroot,0)=\sum_{i=1}^r q^{e_i},$$ where $e_1,\cdots,e_r$ are the exponents of $\mathfrak{g}$. \\

These are only a few of the many open questions in this particular area. In fact, sometimes simply determining whether the multiplicity is positive when using Kostant's weight multiplicity formula is very difficult. Therefore answering questions regarding the support of the multiplicity formula as well as computing closed formulas for the partition function and it's $q$-analog provide many paths for new research.

\section{Appendix} 
To aid in answering some of the open problems listed above, this Appendix provides data on the alternation set and Kostant's partition function for classical Lie algebras.

\begin{table}[htp]
\centering
\caption{Data for Lie algebra of Type $B$ for ranks $2,\;3,\;4,$ and $5.$}
\label{table1B}
\rowcolors{1}{}{midgray}
 \begin{longtable}{|l|l|l|l|l|}
\hline
\rowcolor{lightgray}$\sigma\in\B_r$					&$\ell(\sigma)$		& String		&	$\wp_q(\sigma(\tilde{\alpha}+\rho)-\rho)$					&	$\wp_q|_{q=1}$\\
\hline
\rowcolor{white}\multicolumn{5}{|c|}{\textbf{Rank: $r=2$}}\\
\hline
\rowcolor{white}1						&0		&	$12$				&	$q(q^2+q+1)$				&	3 \\
$s_1$					&1		&	$02$				&	$q^2$					&	1 \\
\hline
\rowcolor{white}\multicolumn{5}{|c|}{$m_q(\tilde\alpha,0)=\sum_{\sigma\in \B_2}(-1)^{\ell(\sigma)}\wp_{q}(\sigma(\tilde\alpha+\rho)-\rho)=q^1+q^3$}\\
\hline 
\hline
\rowcolor{white}\multicolumn{5}{|c|}{\textbf{Rank: $r=3$}}\\
\hline
\rowcolor{white}1						&0		&	$122$			&	$q(q^4+2q^3+4q^2+3q+1)$	&	 11\\
$s_3$					&1		&	$121$			&	$q^2(q^2+2q+2)$			&	 5\\
$s_2$					&1		&	$102$			&	$q^3$					&	 1\\
$s_1$					&1		&	$022$			&	$q^2(q^2+q+2)$				&	 4\\
$s_1s_3$					&2		&	$021$			&	$q^2(q+1)$				&	 2\\

\hline
\rowcolor{white}\multicolumn{5}{|c|}{$m_q(\tilde\alpha,0)=\sum_{\sigma\in \B_3}(-1)^{\ell(\sigma)}\wp_{q}(\sigma(\tilde\alpha+\rho)-\rho)=q^1+q^3+q^5$}\\
\hline 

\hline
\rowcolor{white}\multicolumn{5}{|c|}{\textbf{Rank: $r=4$}}\\
\hline
1						&0		&	$1222$			&	$q(q^6+3q^5+8q^4+11q^3+11q^2+5q+1)$		&	 40\\
$s_1$					&1		&	$0222$			&	$q^2(q^2+q+1)(q^2+q+3)$				&	 15\\
$s_2$					&1		&	$1022$			&	$q^3(q^2+q+2)$						&	 4\\
$s_3$					&1		&	$1212$			&	$q^2(q^2+q+1)(q^2+2q+2)$				&	 15\\
$s_4$					&1		&	$1221$			&	$q^2(q+1)(q^3+2q^2+4q+2)$				&	 18\\
$s_1s_3$					&2		&	$0212$			&	$q^2(q+1)(q^2+q+1)$					&	 6\\
$s_1s_4$					&2		&	$0221$			&	$q^2(q^3+2q^2+3q+1)$					&	 7\\
$s_2s_4$					&2		&	$1021$			&	$q^3(q+1)$							&	 2\\
$s_3s_4$					&2		&	$1201$			&	$q^3(q+1)$							&	 2\\
$s_1s_3s_4$				&3		&	$0201$			&	$q^3$								&	 1\\

\hline
\rowcolor{white}\multicolumn{5}{|c|}{$m_q(\tilde\alpha,0)=\sum_{\sigma\in \B_4}(-1)^{\ell(\sigma)}\wp_{q}(\sigma(\tilde\alpha+\rho)-\rho)=q^1+q^3+q^5+q^7$}\\
\hline 

\hline
\rowcolor{white}\multicolumn{5}{|c|}{\textbf{Rank: $r=5$}}\\
\hline
1						&0		&	$12222$		&	$q(q^8+4q^7+13q^6+25q^5+37q^4+35q^3+22q^2+7q+1)$		&145	 \\
$s_1$					&1		&	$02222$		&	$q^2(q^6+3q^5+9q^4+13q^3+16q^2+9q+4)$					&55	\\
$s_2$					&1		&	$10222$		&	$q^3(q^2+q+1)(q^2+q+3)$								&15	 \\
$s_3$					&1		&	$12122$		&	$q^2(q^2+2q+2)(q^4+2q^3+4q^2+3q+1)$						&55	\\
$s_4$					&1		&	$12212$		&	$q^2(q+1)(q^2+q+1)(q^3+2q^2+4q+2)$						&54	 \\
$s_5$					&1		&	$12221$		&	$q^2(q^2+2q+2)(q^4+2q^3+5q^2+4q+1)$						&65	 \\

$s_1s_3$					&2		&	$02122$		&	$q^2(q+1)(q^4+2q^3+4q^2+3q+1)$							&22	\\
$s_1s_4$					&2		&	$02212$		&	$q^2(q^2+q+1)(q^3+2q^2+3q+1)$							&21	\\	
$s_1s_5$					&2		&	$02221$		&	$q^2(q^5+3q^4+7q^3+8q^2+5q+1)$							&25	\\
$s_2s_4$					&2		&	$10212$		&	$q^3(q+1)(q^2+q+1)$									&6     \\
$s_2s_5$					&2		&	$10221$		&	$q^3(q^3+2q^2+3q+1)$									&7	\\	
$s_3s_4$					&2		&	$12102$		&	$q^4(q^2+2q+2)$										&5	\\
$s_3s_5$					&2		&	$12121$		&	$q^3(q^2+2q+2)^2$										&25	\\			
$s_4s_3$					&2		&	$12012$		&	$q^3(q+1)(q^2+q+1)$									&6     \\
$s_4s_5$					&2		&	$12201$		&	$q^3(q^3+2q^2+3q+1)$									&7	 \\

$s_1s_3s_4$				&3		&	$02102$		&	$q^4(q+1)$											&2	\\							
$s_1s_3s_5$				&3		&	$02121$		&	$q^3(q+1)(q^2+2q+2)$									&10	\\
$s_1s_4s_3$				&3		&	$02012$		&	$q^3(q^2+q+1)$										&3	\\	
$s_1s_4s_5$				&3		&	$02201$		&	$q^3(q^2+q+1)$										&3	\\
$s_2s_4s_5$				&3		&	$10201$		&	$q^4$												&1	\\							
$s_3s_4s_3$				&3		&	$12002$		&	$q^4(q+1)$											&2	\\	

$s_1s_3s_4s_3$			&4		&	$02002$		&	$q^4$												&1	\\
\hline

\rowcolor{white}\multicolumn{5}{|c|}{$m_q(\tilde\alpha,0)=\sum_{\sigma\in \B_5}(-1)^{\ell(\sigma)}\wp_{q}(\sigma(\tilde\alpha+\rho)-\rho)=q^1+q^3+q^5+q^7+q^9$}\\
\hline 

\end{longtable}

\end{table}

\begin{table}
\caption{Data for Lie algebra of Type $B$ of rank $6$.}
\label{table2B}
\rowcolors{1}{}{midgray}
 \begin{longtable}{|l|l|l|l|l|}
\hline
\rowcolor{lightgray}$\sigma\in\B_r$					&$\ell(\sigma)$		& String		&	$\wp_q(\sigma(\tilde{\alpha}+\rho)-\rho)$	&	$\wp_q|_{q=1}$\\
\hline
\rowcolor{white}\multicolumn{5}{|c|}{\textbf{Rank: $r=6$}}\\
\hline
\rowcolor{white}	1						&0		&	$122222$		&	$q(q^{10}+5q^9+19q^8+46q^7+87q^6+118q^5+120q^4+82q^3+37q^2+9q+1)$	&525\\
	$s_1$					&1		&	$022222$		&	$q^2(q^4+2q^3+4q^2+2q+1)(q^4+2q^3+6q^2+6q+5)$					&200	\\		
	$s_2$					&1		&	$102222$		&	$q^3(q^6+3q^5+9q^4+13q^3+16q^2+9q+4)$							&55\\
	$s_3$					&1		&	$121222$		&	$q^2(q^2+2q+2)(q^6+3q^5+8q^4+11q^3+11q^2+5q+1)$					&200\\
	$s_4$					&1		&	$122122$		&	$q^2(q+1)(q^3+2q^2+4q+2)(q^4+2q^3+4q^2+3q+1)$					&198	\\
	$s_5$					&1		&	$122212$		&	$q^2(q^2+q+1)(q^2+2q+2)(q^4+2q^3+5q^2+4q+1)$						&195\\
	$s_6$					&1		&	$122221$		&	$q^2(q^8+5q^7+17q^6+37q^5+58q^4+61q^3+40q^2+14q+2)$				&235	\\

	$s_1s_3$					&2		&	$021222$		&	$q^2(q+1)(q^6+3q^5+8q^4+11q^3+11q^2+5q+1)$						&80\\
	$s_1s_4$					&2		&	$022122$		&	$q^2(q^3+2q^2+3q+1)(q^4+2q^3+4q^2+3q+1)$						&77\\	
	$s_1s_5$					&2		&	$022212$		&	$q^2(q^2+q+1)(q^5+3q^4+7q^3+8q^2+5q+1)$							&75\\
	$s_1s_6$					&2		&	$022221$		&	$q^2(q^7+4q^6+12q^5+21q^4+26q^3+18q^2+7q+1)$					&90\\
	$s_2s_4$					&2		&	$102122$		&	$q^3(q+1)(q^4+2q^3+4q^2+3q+1)$									&22\\
	$s_2s_5$					&2		&	$102212$		&	$q^3(q^2+q+1)(q^3+2q^2+3q+1)$									&21\\
	$s_2s_6$					&2		&	$102221$		&	$q^3(q^5+3q^4+7q^3+8q^2+5q+1)$									&25\\
	$s_3s_4$					&2		&	$121022$		&	$q^4(q^2+q+2)(q^2+2q+2)$										&20\\
	$s_3s_5$					&2		&	$121212$		&	$q^3(q^2+q+1)(q^2+2q+2)^2$										&75\\
	$s_3s_6$					&2		&	$121221$		&	$q^3(q+1)(q^2+2q+2)(q^3+2q^2+4q+2)$								&90\\
	$s_4s_3$					&2		&	$120122$		&	$q^3(q+1)(q^4+2q^3+4q^2+3q+1)$									&22\\
	$s_4s_5$					&2		&	$122102$		&	$q^4(q+1)(q^3+2q^2+4q+2)$										&18\\
	$s_4s_6$					&2		&	$122121$		&	$q^3(q+1)(q^2+2q+2)(q^3+2q^2+4q+2)$								&90\\
	$s_5s_4$					&2		&	$122012$		&	$q^3(q^2+q+1)(q^3+2q^2+3q+1)$									&21\\
	$s_5s_6$					&2		&	$122201$		&	$q^3(q^5+3q^4+7q^3+8q^2+5q+1)$									&25\\

	$s_1s_3s_4$				&3		&	$021022$		&	$q^4(q+1)(q^2+q+2)$											&8\\						
	$s_1s_3s_5$				&3		&	$021212$		&	$q^3(q+1)(q^2+q+1)(q^2+2q+2)$									&30\\
	$s_1s_3s_6$				&3		&	$021221$		&	$q^3(q+1)^2(q^3+2q^2+4q+2)$										&36\\
	$s_1s_4s_3$				&3		&	$020122$		&	$q^3(q^4+2q^3+4q^2+3q+1)$										&11\\	
	$s_1s_4s_5$				&3		&	$022102$		&	$q^4(q^3+2q^2+3q+1)$											&7\\
	$s_1s_4s_6$				&3		&	$022121$		&	$q^3(q^2+2q+2)(q^3+2q^2+3q+1)$									&35\\	
	$s_1s_5s_4$				&3		&	$022012$		&	$q^3(q^2+q+1)^2$												&9\\
	$s_1s_5s_6$				&3		&	$022201$		&	$q^3(q^4+2q^3+4q^2+2q+1)$										&10\\

	$s_2s_4s_5$				&3		&	$102102$		&	$q^5(q+1)$													&2\\
	$s_2s_4s_6$				&3		&	$102121$		&	$q^4(q+1)(q^2+2q+2)$											&10\\
	$s_2s_5s_4$				&3		&	$102012$		&	$q^4(q^2+q+1)$												&3\\
	$s_2s_5s_6$				&3		&	$102201$		&	$q^4(q^2+q+1)$												&3\\
	$s_3s_4s_3$				&3		&	$120022$		&	$q^4(q+1)(q^2+q+2)$											&8\\	
	$s_3s_4s_6$				&3		&	$121021$		&	$q^4(q+1)(q^2+2q+2)$											&10\\
	$s_3s_5s_4$				&3		&	$120102$		&	$q^5(q+1)$													&2\\
	$s_4s_3s_6$				&3		&	$120121$		&	$q^4(q+1)(q^2+2q+2)$											&10\\
	$s_4s_5s_4$				&3		&	$122002$		&	$q^4(q^3+2q^2+3q+1)$											&7\\
	$s_4s_5s_6$				&3		&	$121201$		&	$q^4(q+1)(q^2+2q+2)$											&10\\

	$s_1s_3s_4s_3$			&4		&	$020022$		&	$q^4(q^2+q+2)$												&4\\
	$s_1s_3s_4s_6$			&4		&	$021021$		&	$q^4(q+1)^2$													&4\\						
	$s_1s_3s_5s_4$			&4		&	$020102$		&	$q^5$														&1\\
	$s_1s_4s_3s_6$			&4		&	$020121$		&	$q^4(q^2+2q+2)$												&5\\	
	$s_1s_4s_5s_4$			&4		&	$022002$		&	$q^4(q^2+q+1)$												&3\\
	$s_1s_4s_5s_6$			&4		&	$021201$		&	$q^4(q+1)^2$													&4\\
	$s_2s_4s_5s_4$			&4		&	$102002$		&	$q^5$														&1\\
	$s_3s_4s_3s_6$			&4		&	$120021$		&	$q^4(q+1)^2$													&4\\

	$s_1s_3s_4s_3s_6$			&5		&	$020021$		&	$q^4(q+1)$													&2\\
\hline
\rowcolor{white}\multicolumn{5}{|c|}{$m_q(\tilde\alpha,0)=\sum_{\sigma\in \B_6}(-1)^{\ell(\sigma)}\wp_{q}(\sigma(\tilde\alpha+\rho)-\rho)=q^1+q^3+q^5+q^7+q^9+q^{11}$}\\
\hline 
\end{longtable}
\end{table}

\begin{table}
\caption{Data for Lie algebra of Type $C$ for ranks $2,\;3,\;4,$ and $5.$}
\label{table1C}
\rowcolors{1}{}{midgray}
 \begin{longtable}{|l|l|l|l|l|}
\hline
\rowcolor{lightgray}$\sigma\in\C_r$					&$\ell(\sigma)$		& String		&	$\wp_q(\sigma(\tilde{\alpha}+\rho)-\rho)$					&	$\wp_q|_{q=1}$\\
\hline
\rowcolor{white}\multicolumn{5}{|c|}{\textbf{Rank: $r=2$}}\\
\hline
\rowcolor{white}1						&0		&	$21$				&	$q(q^2+q+1)$				&	3\\
$s_1$					&1		&	$20$				&	$q^2$					&	1 \\
\hline
\rowcolor{white}\multicolumn{5}{|c|}{$m_q(\tilde\alpha,0)=\sum_{\sigma\in \C_2}(-1)^{\ell(\sigma)}\wp_{q}(\sigma(\tilde\alpha+\rho)-\rho)=q^1+q^3$}\\
\hline

\hline
\rowcolor{white}\multicolumn{5}{|c|}{\textbf{Rank: $r=3$}}\\
\hline
\rowcolor{white}1						&0		&	$221$			&	$q(q^4+2q^3+4q^2+2q+1)$	&	10\\
$s_2$					&1		&	$211$			&	$q^2(q+1)^2$				&	 4\\
$s_3$					&1		&	$220$			&	$q^2(q^2+q+1)$			&	3\\
\hline
\rowcolor{white}\multicolumn{5}{|c|}{$m_q(\tilde\alpha,0)=\sum_{\sigma\in \C_3}(-1)^{\ell(\sigma)}\wp_{q}(\sigma(\tilde\alpha+\rho)-\rho)=q^1+q^3+q^5$}\\
\hline
\hline
\rowcolor{white}\multicolumn{5}{|c|}{\textbf{Rank: $r=4$}}\\
\hline
1						&0		&	$2221$			&	$q(q^6+3q^5+8q^4+10q^3+9q^2+3q+1)$		&	 35\\
$s_2$					&1		&	$2121$			&	$q^2(q+1)(q^3+2q^2+3q+1)$				&	 14\\
$s_3$					&1		&	$2211$			&	$q^2(q+1)(q^3+2q^2+3q+1)$				&	 14\\
$s_4$					&1		&	$2220$			&	$q^2(q^4+2q^3+4q^2+2q+1)$				&	 10\\
$s_2s_3$					&2		&	$2011$			&	$q^3(q+1)$							&	 2\\
$s_3s_2$					&2		&	$2101$			&	$q^3(q+1)$							&	 2\\
$s_2s_4$					&2		&	$2120$			&	$q^3(q+1)^2$							&	 4\\
$s_2s_3s_2$				&3		&	$2001$			&	$q^3$								&	 1\\
\hline
\rowcolor{white}\multicolumn{5}{|c|}{$m_q(\tilde\alpha,0)=\sum_{\sigma\in \C_4}(-1)^{\ell(\sigma)}\wp_{q}(\sigma(\tilde\alpha+\rho)-\rho)=q^1+q^3+q^5+q^7$}\\
\hline

\hline
\rowcolor{white}\multicolumn{5}{|c|}{\textbf{Rank: $r=5$}}\\
\hline
1						&0		&	$22221$		&	$(q^8+4q^7+13q^6+24q^5+34q^4+28q^3+16q^2+4q+1)q$	&125	 \\
$s_2$					&1		&	$21221$		&	$(q+1)(q^5+3q^4+7q^3+8q^2+5q+1)q^2$					&50	 \\
$s_3$					&1		&	$22121$		&	$(q^3+2q^2+3q+1)^2q^2$							&49	 \\
$s_4$					&1		&	$22211$		&	$(q+1)(q^5+3q^4+7q^3+8q^2+5q+1)q^2$					&50	 \\
$s_5$					&1		&	$22220$		&	$(q^6+3q^5+8q^4+10q^3+9q^2+3q+1)q^2$					&35	\\
$s_2s_3$					&2		&	$20121$		&	$(q^3+2q^2+3q+1)q^3$								&7	\\		
$s_3s_2$					&2		&	$21021$		&	$(q+1)(q^2+q+1)q^3$									&6     \\
$s_2s_5$					&2		&	$21220$		&	$(q+1)(q^3+2q^2+3q+1)q^3$								&14	 \\
$s_2s_4$					&2		&	$21211$		&	$(q^2+2q+2)(q+1)^2q^3$									&20	\\			
$s_3s_4$					&2		&	$22011$		&	$(q+1)(q^2+q+1)q^3$									&6	\\
$s_4s_3$					&2		&	$22101$		&	$(q^3+2q^2+3q+1)q^3$									&7	\\
$s_3s_5$					&2		&	$22120$		&	$(q+1)(q^3+2q^2+3q+1)q^3$								&14	 \\

$s_2s_3s_5$				&3		&	$20120$		&	$(q+1)q^4$											&2	\\
$s_3s_2s_5$				&3		&	$21020$		&	$(q+1)q^4$											&2	\\	
$s_2s_3s_2$				&3		&	$20021$		&	$(q^2+q+1)q^3$										&3	\\	
$s_3s_4s_3$				&3		&	$22001$		&	$(q^2+q+1)q^3$										&3	\\	
$s_2s_4s_3$				&3		&	$20101$		&	$q^4$												&1	\\

$s_2s_3s_2s_5$			&4		&	$20020$		&	$q^4$												&1    \\
\hline
\rowcolor{white}\multicolumn{5}{|c|}{$m_q(\tilde\alpha,0)=\sum_{\sigma\in \C_5}(-1)^{\ell(\sigma)}\wp_{q}(\sigma(\tilde\alpha+\rho)-\rho)=q^1+q^3+q^5+q^7+q^9$}\\

\hline
\end{longtable}

\end{table}

\begin{table}
\caption{Data for Lie algebra of Type $C$ of rank $6$.}
\label{table2C}
\centering

\rowcolors{1}{}{midgray}
\begin{longtable}{|l|l|l|l|l|}
\hline
\rowcolor{lightgray}$\sigma\in\C_r$					&$\ell(\sigma)$		& String		&	$\wp_q(\sigma(\tilde{\alpha}+\rho)-\rho)$					&	$\wp_q|_{q=1}$\\
\hline

\rowcolor{white}\multicolumn{5}{|c|}{\textbf{Rank: $r=6$}}\\
\hline
\rowcolor{white}	1				&0		&	$222221$ & $(q^{10}+5q^9+19q^8+45q^7+83q^6+106q^5+100q^4+60q^3+25q^2+5q+1)q$ 	&450 \\
	$s_2$			&1		&	$212221$ & $(q+1)(q^7+4q^6+12q^5+21q^4+26q^3+18q^2+7q+1)q^2$	&180 \\
	$s_3$			&1		&	 $221221$ & $(q^3+2q^2+3q+1)(q^5+3q^4+7q^3+8q^2+5q+1)q^2$ 	&175 \\
	$s_4$			&1		&	 $222121$ & $(q^3+2q^2+3q+1)(q^5+3q^4+7q^3+8q^2+5q+1)q^2$ 	&175\\
	$s_5$			&1		&	$222211$ & $(q+1)(q^7+4q^6+12q^5+21q^4+26q^3+18q^2+7q+1)q^2$ 	&180\\
	$s_6$			&1		&	$222220$ & $0q^1+1q^2+4q^3+16q^4+28q^5+34q^6+24q^7+13q^8+4q^9+1q^10$	&125\\
	$s_2s_3$			&2		&	$201221$ & $(q^5+3q^4+7q^3+8q^2+5q+1)q^3$ 	&25\\
	$s_2s_4$			&2		&	$212121$ & $(q+1)(q^2+2q+2)(q^3+2q^2+3q+1)q^3$ 	&70\\
	$s_2s_5$			&2		&	 $212211$ & $(q^3+2q^2+4q+2)(q+1)^3q^3$ &72\\
	$s_2s_6$			&2		&	$212220$	&  $(q+1)(q^5+3q^4+7q^3+8q^2+5q+1)q^3$ 	&50\\
	$s_3s_2$			&2		&	$210221$ & $(q+1)(q^4+2q^3+4q^2+2q+1)q^3$ 	&20\\
	$s_3s_4$			&2		&	$220121$ & $(q^2+q+1)(q^3+2q^2+3q+1)q^3$ 	&21\\	
	$s_3s_5$			&2		&	$221211$ & $(q+1)(q^2+2q+2)(q^3+2q^2+3q+1)q^3$ 	&70\\	
	$s_3s_6$			&2		&	$221220$ & $(q^3+2q^2+3q+1)^2q^3$	&49\\
	$s_4s_3$			&2		&	$221021$ & $(q^2+q+1)(q^3+2q^2+3q+1)q^3$	&21\\
	$s_4s_5$			&2		&	$222011$ & $(q+1)(q^4+2q^3+4q^2+2q+1)q^3$	&20\\		
	$s_4s_6$			&2		&	$222120$ & $(q+1)(q^5+3q^4+7q^3+8q^2+5q+1)q^3$	&50\\
	$s_5s_4$			&2		&	$222101$ & $(q^5+3q^4+7q^3+8q^2+5q+1)q^3$	&25\\

	$s_2s_3s_2$		&3		&	$200221$	&  $(q^4+2q^3+4q^2+2q+1)q^3$ 	&10\\
	$s_2s_3s_5$		&3		&	$201211$ & $(q+1)(q^2+2q+2)q^4$ &10\\
	$s_2s_3s_6$		&3		&	$201220$ & $(q^3+2q^2+3q+1)q^4$ 	&7\\
	$s_2s_4s_5$		&3		&	$212011$ & $(q+1)^3q^4$ 	&8\\
	$s_2s_4s_6$		&3		&	$212120$ & $(q^2+2q+2)(q+1)^27q^4$ 	&20\\
	$s_2s_5s_4$		&3		&	$212101$ & $(q+1)(q^2+2q+2)q^4$ 	&10\\
	$s_3s_4s_3$		&3		&	$220021$ & $(q^2+q+1)^2q^3$ 	&9\\
	$s_2s_4s_3$		&3		&	$201021$ & $(q^2+q+1)q^4$ 	&3\\
	$s_3s_2s_5$		&3		&	$210211$ & $(q+1)^3q^4$	&8\\
	$s_3s_4s_6$		&3		&	$220120$ & $(q+1)(q^2+q+1)q^4$	&6\\
	$s_3s_2s_6$		&3		&	$210220$ & $(q+1)(q^2+q+1)q^4$	&6\\
	$s_4s_5s_4$		&3		&	$222001$ & $(q^4+2q^3+4q^2+2q+1)q^3$	&10\\
	$s_3s_5s_4$		&3		&	$220101$ & $(q^2+q+1)q^4$	&3\\
	$s_4s_3s_6$		&3		&	$221020$ & $(q^3+2q^2+3q+1)q^4$	&7\\

	$s_2s_3s_2s_6$	&4		&	$200220$ & $(q^2+q+1)q^4$ 	&3\\
	$s_2s_3s_2s_5$	&4		&	$200211$ & $(q+1)^2q^4$ 	&4\\
	$s_2s_4s_5s_4$	&4		&	$212001$ & $(q+1)^2q^4$	&4\\
	$s_3s_4s_3s_6$	&4		&	$220020$ & $(q^2+q+1)q^4$	&3\\
	$s_2s_4s_3s_6$	&4		&	$201020$ & $q^5$	&1\\
	
	\hline
\rowcolor{white}\multicolumn{5}{|c|}{$m_q(\tilde\alpha,0)=\sum_{\sigma\in \C_6}(-1)^{\ell(\sigma)}\wp_{q}(\sigma(\tilde\alpha+\rho)-\rho)=q^1+q^3+q^5+q^7+q^9+q^{11}$}\\

	\hline
\end{longtable}

\end{table}

\begin{table}
\caption{Data for Lie algebra of Type $D$ for ranks $4$ and $5$.}
\label{table1D}
\rowcolors{1}{}{midgray}
 \begin{longtable}{|l|l|l|l|l|}
\hline
\rowcolor{lightgray}$\sigma\in\D_r$					&$\ell(\sigma)$		& String		&	$\wp_q(\sigma(\tilde{\alpha}+\rho)-\rho)$	&	$\wp_q|_{q=1}$\\
\hline
\hline
\rowcolor{white}\multicolumn{5}{|c|}{\textbf{Rank: $r=4$}}\\
\hline
\rowcolor{white}1						&0		&	$1211$			&	$(q^4+3q^3+6q^2+4q+1)q$		&15	\\
$s_1$					&1		&	$0211$			&	$(q^2+2q+2)q^2$		&5	\\
$s_2$					&1		&	$1011$			&	$q^3$		&1	 \\
$s_3$					&1		&	$1201$			&	$(q^2+2q+2)q^2$		&5	 \\
$s_4$					&1		&	$1210$			&	$(q^2+2q+2)q^2$		&5	\\
$s_1s_3$					&2		&	$0201$			&	$(q+1)q^2$		&2	 \\
$s_1s_4$					&2		&	$0210$			&	$(q+1)q^2$		&2	 \\
$s_3s_4$					&2		&	$1200$			&	$(q+1)q^2$		&2	 \\
$s_1s_3s_4$				&3		&	$0200$			&	$q^2$		&1	 \\
\hline
\rowcolor{white}\multicolumn{5}{|c|}{$m_q(\tilde\alpha,0)=\sum_{\sigma\in \D_4}(-1)^{\ell(\sigma)}\wp_{q}(\sigma(\tilde\alpha+\rho)-\rho)=q^1+2q^3+q^5$}\\
\hline

\hline
\rowcolor{white}\multicolumn{5}{|c|}{\textbf{Rank: $r=5$}}\\
\hline
$1$ 			&0		& $ 1 2 2 1 1$ 		& $(q^6+4q^5+11q^4+17q^3+15q^2+6q+1)q$  &55\\
$s_1$ 		&1		& $ 0 2 2 1 1$ 		& $(q^4+3q^3+7q^2+6q+3)q^2$  			&20\\
$s_2$	 	&1		& $ 1 0 2 1 1$ 		& $(q^2+2q+2)q^3$  						&5\\
$s_3$ 		&1		& $ 1 2 1 1 1$ 		& $(q^2+2q+2)(q+1)^2q^2$  				&20\\
$s_4$ 		&1		& $ 1 2 2 0 1$		& $(q+1)(q^3+2q^2+4q+2)q^2$  			&18\\
$s_5$ 		&1		& $ 1 2 2 1 0$ 		& $(q+1)(q^3+2q^2+4q+2)q^2$  			&18\\
$s_3s_1$ 		&2		& $ 0 2 1 1 1$ 		& $(q+1)^3q^2$  						&8\\
$s_3s_4$ 		&2		& $ 1 2 0 0 1$ 		& $(q+1)q^3$  							&2\\
$s_3s_5$  	&2		& $ 1 2 0 1 0$ 		& $(q+1)q^3$  							&2\\
$s_4s_1$ 		&2		& $ 0 2 2 0 1$ 		& $(q^3+2q^2+3q+1)q^2$  				&7\\
$s_4s_2$ 		&2		& $ 1 0 2 0 1$ 		& $(q+1)q^3$  							&2\\
$s_5s_1$ 		&2		& $ 0 2 2 1 0$ 		& $(q^3+2q^2+3q+1)q^2$  				&7\\
$s_5s_2$ 		&2		& $ 1 0 2 1 0$ 		& $(q+1)q^3$  							&2\\
$s_5s_4$ 		&2		& $ 1 2 2 0 0$ 		& $(q^3+2q^2+3q+1)q^2$  				&7\\
$s_3s_4s_1$ 	&3		& $ 0 2 0 0 1$ 		& $q^3$ 		 						&1\\
$s_3s_5s_1$ 	&3		& $ 0 2 0 1 0$ 		& $q^3$  								&1\\
$s_5s_4s_1$ 	&3		& $ 0 2 2 0 0$ 		& $(q^2+q+1)q^2$  						&3\\
$s_5s_4s_2$ 	&3		& $ 1 0 2 0 0$ 		& $q^3$  								&1\\

\hline
\rowcolor{white}\multicolumn{5}{|c|}{$m_q(\tilde\alpha,0)=\sum_{\sigma\in \D_5}(-1)^{\ell(\sigma)}\wp_{q}(\sigma(\tilde\alpha+\rho)-\rho)=q^1+q^3+q^4+q^5+q^7$}\\

\hline
\end{longtable}

\end{table}

\begin{table}
\caption{Data for Lie algebra of Type $D$ of rank $6$.}
\label{table2D}
\rowcolors{1}{}{midgray} 
\begin{longtable}{|l|l|l|l|l|}
\hline
\rowcolor{lightgray}$\sigma\in\D_r$					&$\ell(\sigma)$		& String		&	$\wp_q(\sigma(\tilde{\alpha}+\rho)-\rho)$	&	$\wp_q|_{q=1}$\\
\hline
\rowcolor{white}\multicolumn{5}{|c|}{\textbf{Rank: $r=6$}}\\
\hline

\rowcolor{white}$1$ 				&0	& $1 2 2 2 1 1$	& $(q^8+5q^7+17q^6+36q^5+54q^4+50q^3+28q^2+8q+1)q$ &200\\
$s_1$ 			&1	& $0 2 2 2 1 1$ & $(q^2+2q+2)(q^4+2q^3+6q^2+4q+2)q^2$ &75\\
$s_2$ 			&1	& $1 0 2 2 1 1$ & $(q^4+3q^3+7q^2+6q+3)q^3$ &20\\
$s_3$ 			&1	& $1 2 1 2 1 1$ & $(q^2+2q+2)(q^4+3q^3+6q^2+4q+1)q^2$ &75\\
$s_4$ 			&1	& $1 2 2 1 1 1$ & $(q^3+2q^2+4q+2)(q+1)^3q^2$ &72\\
$s_5$ 			&1	& $1 2 2 2 0 1$ & $(q^2+2q+2)(q^4+2q^3+5q^2+4q+1)q^2$ &65\\
$s_6$ 			&1	& $1 2 2 2 1 0$ & $(q^2+2q+2)(q^4+2q^3+5q^2+4q+1)q^2$ &65\\
$s_3s_1$ 			&2	& $0 2 1 2 1 1$ & $(q+1)(q^4+3q^3+6q^2+4q+1)q^2$ &30\\
$s_3s_4$ 			&2	& $1 2 0 1 1 1$ & $(q+1)^3q^3$ &8\\
$s_4s_1$ 			&2	& $0 2 2 1 1 1$ & $(q^3+2q^2+3q+1)(q+1)^2q^2$ &28\\
$s_4s_2$ 			&2	& $1 0 2 1 1 1$ & $(q+1)^3q^3$ &8\\
$s_4s_3$ 			&2	& $1 2 1 0 1 1$ & $(q^2+2q+2)q^4$ &5\\
$s_4s_5$ 			&2	& $1 2 2 0 0 1$ & $(q^3+2q^2+3q+1)q^3$ &7\\
$s_4s_6$ 			&2	& $1 2 2 0 1 0$ & $(q^3+2q^2+3q+1)q^3$ &7\\
$s_5s_1$ 			&2	& $0 2 2 2 0 1$ & $(q^5+3q^4+7q^3+8q^2+5q+1)q^2$ &25\\
$s_5s_2$ 			&2	& $1 0 2 2 0 1$ & $(q^3+2q^2+3q+1)q^3$ &7\\
$s_5s_3$ 			&2	& $1 2 1 2 0 1$ & $(q^2+2q+2)^2q^3$ &25\\
$s_6s_1$ 			&2	& $0 2 2 2 1 0$ & $(q^5+3q^4+7q^3+8q^2+5q+1)q^2$ &25\\
$s_6s_2$ 			&2	& $1 0 2 2 1 0$ & $(q^3+2q^2+3q+1)q^3$ &7\\
$s_6s_3$ 			&2	& $1 2 1 2 1 0$ & $(q^2+2q+2)^2q^3$ &25\\
$s_6s_5$ 			&2	& $1 2 2 2 0 0$ & $(q^5+3q^4+7q^3+8q^2+5q+1)q^2$ &25\\
$s_3s_4s_1$ 		&3	& $0 2 0 1 1 1$ & $(q+1)^2q^3$ &4\\
$s_3s_4s_3$ 		&3	& $1 2 0 0 1 1$ & $(q+1)q^4$ &2\\
$s_4s_3s_1$ 		&3	& $0 2 1 0 1 1$ & $(q+1)q^4$ &2\\
$s_4s_5s_1$ 		&3	& $0 2 2 0 0 1$ & $(q^2+q+1)q^3$ &3\\
$s_4s_5s_2$ 		&3	& $1 0 2 0 0 1$ & $q^4$ &1\\
$s_4s_6s_1$ 		&3	& $0 2 2 0 1 0$ & $(q^2+q+1)q^3$ &3\\
$s_4s_6s_2$ 		&3	& $1 0 2 0 1 0$ & $q^4$ &1\\
$s_5s_3s_1$ 		&3	& $0 2 1 2 0 1$ & $(q+1)(q^2+2q+2)q^3$ &10\\
$s_6s_3s_1$ 		&3	& $0 2 1 2 1 0$ & $(q+1)(q^2+2q+2)q^3$ &10\\
$s_6s_5s_1$ 		&3	& $0 2 2 2 0 0$ & $(q^4+2q^3+4q^2+2q+1)q^2$ &10\\
$s_6s_5s_2$ 		&3	& $1 0 2 2 0 0$ & $(q^2+q+1)q^3$ &3\\
$s_6s_5s_3$ 		&3	& $1 2 1 2 0 0$ & $(q+1)(q^2+2q+2)q^3$ &10\\
$s_3s_4s_3s_1$ 	&4	& $0 2 0 0 1 1$ & $q^4$ &1	\\
$s_6s_5s_3s_1$ 	&4	& $0 2 1 2 0 0$ & $(q+1)^2q^3$ &4\\
		\hline
\rowcolor{white}\multicolumn{5}{|c|}{$m_q(\tilde\alpha,0)=\sum_{\sigma\in \D_6}(-1)^{\ell(\sigma)}\wp_{q}(\sigma(\tilde\alpha+\rho)-\rho)=q^1+q^3+2q^5+q^7+q^9$}\\
\hline
\end{longtable}
\end{table}

\pagebreak

\end{document}